\documentclass[a4paper, 11pt]{article}

\usepackage{amsmath}
\usepackage{amsthm}
\usepackage{amsfonts}
\usepackage{amssymb}
\usepackage{amscd}
\usepackage[dvips]{graphicx, color}
\newtheorem{thm}{Theorem}[section]

\newtheorem{lem}[thm]{Lemma}
\newtheorem{prop}[thm]{Proposition}
\newtheorem{rem}[thm]{Remark}
\newtheorem{exam}[thm]{Example}

\numberwithin{equation}{section}

\begin{document}
\title{$q$-Deformations and $t$-Deformations of  the Markov triples  }
\author{
{Takeyoshi Kogiso}\\
{\footnotesize  Department of Mathematics, Josai University, }\\
{\footnotesize  1-1, Keyakidai Sakado, Saitama, 350-0295, Japan}\\
{\footnotesize  E-mail address: kogiso@josai.ac.jp}\\  
}
\date{}

\maketitle

\begin{abstract}  
 In the present paper, we generalize the Markov triples in two different directions. 
 One is generalization in direction of using the $q$-deformation of rational number introduced by \cite{MO} in connection with cluster algebras, quantum topology and analytic number theory. The other is a generalization in direction of using castling transforms on prehomogeneous vector spaces \cite{SaKi} which plays an important role in the study of representation theory and  automorphic functions. In addition, the present paper gives a relationship between the two generalizations. 
 This may provide some kind of bridging between different fields.
 \end{abstract}

\section*{Introduction}
It is well known that Markov triple $(x,y,z)=(\mathrm{Tr}(w(A,B))/3, \mathrm{Tr}(w(A,B) w'(A,B))/3, \mathrm{Tr}(w'(A,B))/3)$ as a solution of Markov equation 
$x^{2} + y^{2} +z^{2} =3xyz$ for a triple of Christoffel $ab$-word $(w(a,b), w(a,b)w'(a,b), w' (a,b))$, where 
$A=\begin{pmatrix} 
2 & 1 \\ 
1 & 1 
\end{pmatrix} , ~
B=\begin{pmatrix} 
5 & 2 \\
2 & 1 
\end{pmatrix}$.

\noindent 
 This paper introduces two generalizations of Markov triple in two completely different directions and the relation between them.

\noindent 
Sophie Morier-Genoud and Valentine Ovsienko \cite{MO}  introduced $q$-deformation of fractions as the followings. For a positive rational number  $\frac{r}{s}=[a_1, \dots , a_{2m}]=[[c_1, \dots , c_r]]$, they defined $[a_1, \dots , a_{2m}]_q $ and $[c_1, \dots , c_r]_q$ as $q$-deformation respectively. Furthermore they proved the former coincides with the latter. Thus they defined $[\frac{r}{s}]_q :=[a_1, \dots , a_{2m}]_q =[[c_1, \dots , c_r]]_q$ as $q$-deformation of $\frac{r}{s}=[a_1, \dots , a_{2m}]_=[[c_1, \dots , c_r]]$.

\noindent 
Our first generalization of Markov triples $(h_{w(a,b)} (q), h_{w(a,b)w'(a,b)} (q) , h_{w'(a,b)} (q) )$ are given by the properties of $q$-deformation of a positive fraction.
$(x,y,z)=(h_{w(a,b)} (q), h_{w(a,b)w'(a,b)} (q), h_{w'(a,b)}) =(\mathrm{Tr}(w(A_{q}, B_{q})/[3]_{q}, 
\mathrm{Tr}(w(A_{q}, B_{q})w'(A_{q}, B_{q}))/[3]_{q}, \mathrm{Tr}( w' (A_{q}, B_{q} )/[3]_{q})$ is a solution of $q$-deformed Markov equation 
$$x^{2} + y^{2} +z^{2} +\frac{(q-1)^{2} }{q^{3} } =[3]_{q}xyz$$
where $A_{q} =
\begin{pmatrix}
q+1 & q^{-1} \\ 
1 & q^{-1}  
\end{pmatrix} , ~
B_{q} =
\begin{pmatrix} 
\frac{ q^3  + q^2  +2q + 1}{q} & \frac{q+1}{q} \\
\frac{q+1}{q} & q^{-2} 
\end{pmatrix}  \in SL(2,{\Bbb Z}[q, q^{-1}])$.

\noindent 
 On the other hand, Mikio Sato  \cite{SaKi} introduced Theory of prehomogeneous vector spaces and classified irreducible prehomogeneous vector spaces with Tatsuo Kimura by using castling transformations of prehomogeneous vector spaces. Our second generalization of Markov triples are given by an application of castling transformations of $t$-dimensional prehomogeneous vector spaces. We find that a certain subtree of the case for $t=3$ is related to Markov tree of  Markov triples $(\mathrm{Tr}(w(A,B))/3, \mathrm{Tr} (w(A,B) w' (A,B))/3 , \mathrm{Tr}(w' (A,B))/3)$ and show that subtree in  tree of castling transformation of $t$-dimensional prehomogeneous vector space corresponds to a tree of triplets $(f_w (t) , f_{ww'} (t), f_{w'}(t))$ of polynomials associated to a triple of Christoffel $ab$-words $(w(a,b), w(a,b) w'(a,b), w'(a,b))$. This triple $(x,y,z)=(f_{w(a,b)} (t) , f_{w(a,b)w'(a,b)} (t), f_{w'(a,b)}(t))$ is a solution of a $t$-deformed Markov equation $$x^{2} +y^{2} +z^{2} +(t-3) =txyz.$$

\noindent
Moreover there exists a relation $f_{w(a,b)} (q^{-1} [3]_q ) =q h_{w(a,b)} (q)$ between $q$-deformations  and $t$-deformations.

\indent 
This paper is organized into the following sections.

\indent 
 In $\S 1$, properties of $q$-deformation of continued fractions along \cite{MO}
 is described.

\indent 
  In $\S 2$, by using $\S 1$, we defined $q$-deformation of Markov triples and introduce properties and application of them. 
  
 \indent 
 In $\S3$, from the view point of castling transformation of prehomogeneous vector spaces, we define castling Markov triples (=$t$-deformations of Markov triples) and introduce properties of them.
 
 \indent 
 In $\S 4$, relation between the $q$-deformations and the $t$-deformation is described.


\section{$q$-Deformation of continued fractions due to Morier-Genoud and Ovsienko }

\noindent 
For a rational number $\frac{r}{s} \in {\Bbb Q}$ and that $r,s$ are positive coprime integers. 
It is well known that $\frac{r}{s}$ has different continued fraction expansions as follows:


\begin{equation}\label{eqn-1}
\begin{array}{lll}
\frac{r}{s}& = 
a_1 +\cfrac{1}{a_2 +\cfrac{1}{\ddots +\cfrac{1}{a_{2m}}}} 
& =[a_1, a_2, \dots , a_{2m}]  \\  & =c_1 -\cfrac{1}{c_2 -\cfrac{1}{\ddots -\cfrac{1}{c_k}}}  & =   [[c_1, c_2, \dots , c_k]]
\end{array}
\end{equation}

\noindent 
with $c_i \geq 2 $ and $ a_i \geq 1$, denoted by $[a_1, \dots , a_{2m}]$ and $[[c_1, \dots , c_k]]$, respectively. They are usually called regular and negative continued fractions, respectively. Considering an even number of coefficients in the regular expansion and coefficients greater than 2 in the negative expansion make the expansions unique. 
Sophie Morier-Genord and Valentin Ovsienko \cite{MO} introduced 
$q$-deformed rationals and $q$-continued fractions as follows:

\noindent 
Let $q$ be a formal parameter, put 
\begin{equation}\label{eqn-1.5}
[a]_p :=\frac{1-q^a}{1-q}
\end{equation}
where $a$ is a non-negative integer.

\noindent 
(a) 

\begin{equation}\label{eqn-2}
[a_1, a_2, \dots , a_{2m}]_q :=
[a_1]_q  
+\cfrac{ q^{a_1}   }{ [a_2]_{ q^{-1} }  + \cfrac{ q^{-a_2}  }
{  [a_3]_q  + \cfrac{q^{a_3} }
{ [a_4]_{q^{-1} } + \cfrac{ q^{-a_4} }{  \cfrac{\ddots }
{ [a_{2m-1}]_q+ \cfrac{ q^{a_{2m-1}} }{ [a_{2m}]_{q^{-1}}  }}}}}} 
\end{equation}

\noindent 
(b) 
\begin{equation}\label{eqn-3}
[[c_1, c_2, \dots , c_{k}]]_q :=
[c_1]_q  
-\cfrac{ q^{c_1 -1}   }{ [c_2]_{ q }  - \cfrac{ q^{c_2 -1}  }
{  [c_3]_q  - \cfrac{q^{c_3-1 }}
{ [c_4]_{q } - \cfrac{ q^{c_4 -1} }{  \cfrac{\ddots }
{ [c_{k-1}]_q- \cfrac{ q^{c_{k-1}-1} }{ [c_{k}]_{q}  }}}}}} 
\end{equation}

\begin{thm}(Morier-Genoud and Ovsienko \cite{MO}) \label{theorem:MO-qfrac}
If a rational number $\frac{r}{s}$ has a regular  and negative continued fraction expansion form 
\noindent 
$\frac{r}{s}=[a_{1}, \dots , a_{2m}]=[[c_{1}, \dots , c_{k}]]$,

\noindent 
then 
\begin{equation}\label{eqn-1}
[a_{1}, \dots , a_{2m}]_{q}=[[c_{1}, \dots , c_{k}]]_{q} =:[\frac{r}{s}]_q
\end{equation}

\end{thm}

\medskip 

\begin{exam}
$[ \frac{5}{2}]_q =[[3,2]]_q =[2,2]_q =\frac{1+2q +q^2 +q^3}{1+q} $

\noindent 
$[ \frac{5}{3}]_q =[[2,3]]_q =[1,1,1,1]_q =\frac{1+q +2q^2 +q^3}{1+q+q^2} $

\noindent 
$[ \frac{7}{3}]_q =[[3,2,2]]_q =[2,3]_q =\frac{1+2q +2q^2 +q^3+q^4}{1+q+q^2} $

\noindent 
$[ \frac{7}{4}]_q =[[2,4]]_q =[1,1,2,1]_q =\frac{1+q +2q^2 +2q^3+q^4}{1+q+q^2 +q^3} $

\noindent 
$[ \frac{7}{5}]_q =[[2,2,3]]_q =[1,1,2,1]_q =\frac{1+q +2q^2 +2q^3+q^4}{1+q+2q^2 +q^3} $
\end{exam}

\begin{thm}(Morier-Genoud and Ovsienko \cite{MO})\label{theorem:MO-prperty1}
The following two matrices 

\begin{equation}\label{eqn-6}
M^+_q ( a_1, \dots , a_{2m}):=
\begin{pmatrix} 
[a_1]_q & q^{a_1} \\
1 & 0 
\end{pmatrix} 
\begin{pmatrix} 
[a_2]_{q^{-1}} & q^{-a_2 } \\
1 & 0 
\end{pmatrix} 
\cdots 
\begin{pmatrix} 
[a_{2m-1}]_q & q^{a_{2m-1} } \\
1 & 0 
\end{pmatrix} 
\begin{pmatrix} 
[a_{2m}]_{q^{-1}} & -q^{-a_{2m} } \\
1 & 0 
\end{pmatrix} 
\end{equation}

\begin{equation}\label{eqn-7}
M_q ( c_1, \dots , c_k):=
\begin{pmatrix} 
[c_1]_q & -q^{c_1 -1}  \\
1 & 0 
\end{pmatrix} 
\begin{pmatrix} 
[c_2]_q & -q^{c_2 -1} \\
1 & 0 
\end{pmatrix} 
\cdots 
\begin{pmatrix} 
[c_k-1]_q & -q^{c_{k-1} -1} \\
1 & 0 
\end{pmatrix} 
\begin{pmatrix} 
[c_k]_q & -q^{c_k -1} \\
1 & 0 
\end{pmatrix} 
\end{equation}

\noindent 
satisfy the following equations:

\noindent 
(i)
\begin{equation}\label{eqn-8}
M_q ^+ (a_1, \dots , a_{2m})=\begin{pmatrix}
q \mathcal{R} & \mathcal{R}'_{2m-1} \\
q \mathcal{S} & \mathcal{S}'_{2m-1} 
\end{pmatrix}
\end{equation}

\noindent 
where $\frac{\mathcal{R}(q)}{\mathcal{S}(q)}=[a_1, a_2, \dots , a_{2m}]_q , ~~
\frac{\mathcal{R}'_{2m-1} (q)}{\mathcal{S}'_{2m-1}(q)}=[a_1, \dots , a_{2m-1}]_q$

\medskip

\noindent 
(ii)
\begin{equation}\label{eqn-9}
M_q (c_1, \dots , c_k)=
\begin{pmatrix}
 \mathcal{R} & -q^{c_k -1}\mathcal{R}_{k-1} \\
 \mathcal{S} & -q^{c_k -1}\mathcal{S}'_{k-1} 
\end{pmatrix}
\end{equation}

\noindent 
where $\frac{\mathcal{R}(q)}{\mathcal{S}(q)}=[[c_1, \dots , c_{k}]]_q , ~~
\frac{\mathcal{R}_{k-1} (q)}{\mathcal{S}_{k-1}(q)}=[c_1, \dots , c_{k-1}]_q$
\end{thm}

\medskip

\begin{thm}(Morier-Genoud and Ovsienko \cite{MO})
$R_q :=\begin{pmatrix} 
q & 1 \\
0 & 1 
\end{pmatrix} , ~L_q :=\begin{pmatrix} 
1 & 0 \\
1 & q^{-1} 
\end{pmatrix} , ~S_q :=\begin{pmatrix} 
0 & -q^{-1} \\ 
1 & 0 
\end{pmatrix} $, then

\begin{equation}\label{eqn-1}
M_q^+ (a_1 , \dots , a_{2m}) =R_q^{a_1 } L_q^{a_2} \cdots R_q^{a_{2m-1}} L_q^{a_{2m}}
\end{equation}

\noindent 
 and

\begin{equation}\label{eqn-1}
M_q^+ (c_1 , \dots , c_k ) =R_q^{c_1} S_q  R_q^{c_2} S_q  \cdots S_q  R_q^{c_k } S_q .
\end{equation}

\end{thm}

\section{\bf $q$-Deformations of the Markov equation  and Markov triples}

\noindent 
A triple $(x, y,z)$ of positive integers is called Markov triple (up to permutation) when $(x,y , z)$ is a solution of 
the Markov equation 
\begin{equation}
x^2 +y^2 +z^2 =3xyz.
\end{equation}

\noindent 
A well-defined set of representative is obtained starting with $(1,1,1),(1,2,1),(1,5,2)$ and then proceeding recursively going from Markov triple $(p,q,r)$ to the new Markov triples $(3qr-p,q,r), (p,3pr-q,r), (p,q,3pq-r)$(up to permutation). The Markov tree of representatives of Markov triples (Figure 1) and the tree of triples of Christoffel $ab$-words (Figure2) are well known( cf. \cite{Aig}).

\begin{figure}[htbp]
\centering 
\includegraphics[width=13cm]{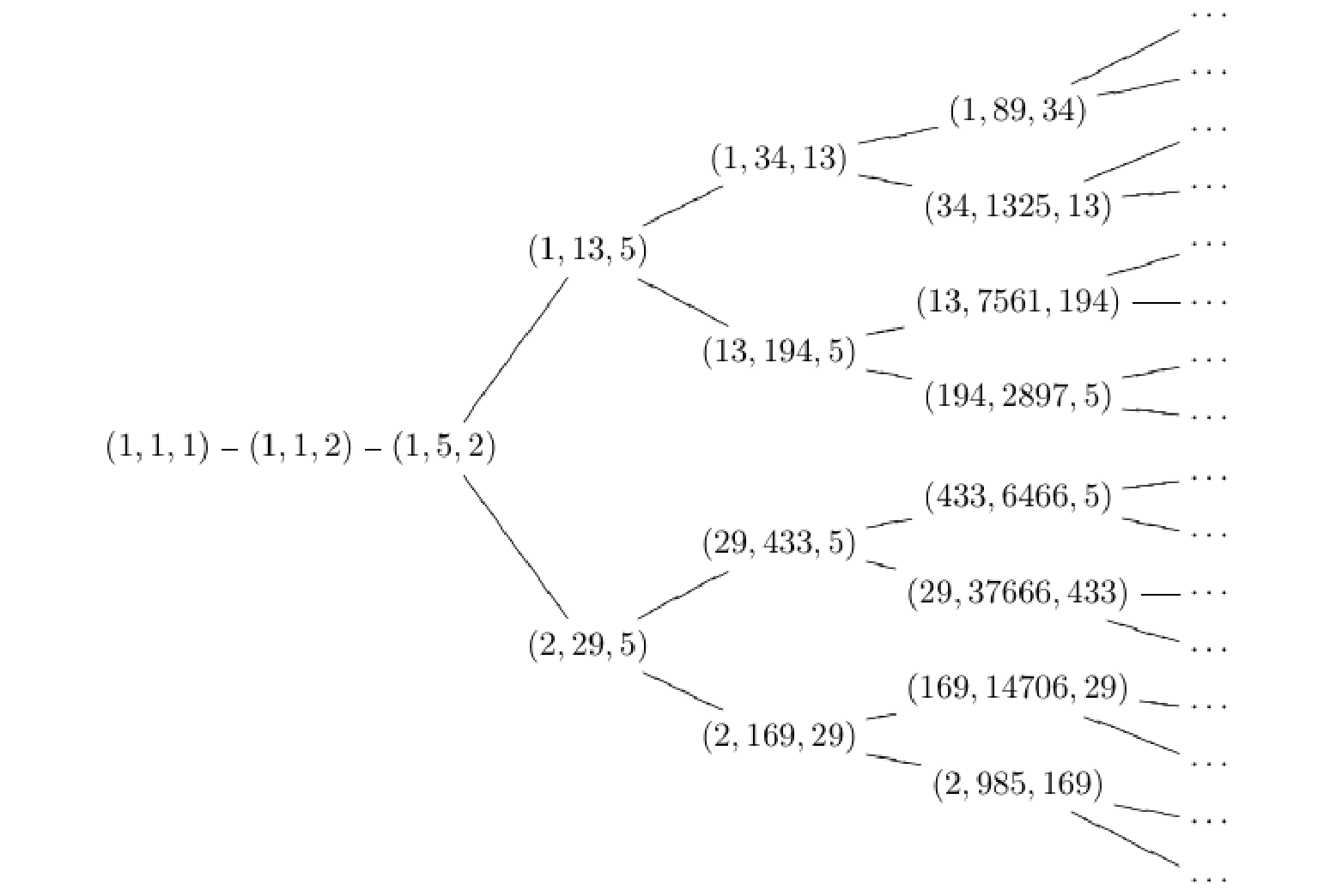} 
\caption{}\label{fig1}
\end{figure}

\noindent 
We consider the following tree of triple of Christoffel $ab$-words ({\it Cohn word} ) form P.201 in \cite{Bom}  corresponding to Markov triples along Bombieri \cite{Bom} and its modification-version in \cite{Aig}

\begin{figure}[htbp]
\centering 
\includegraphics[width=13cm]{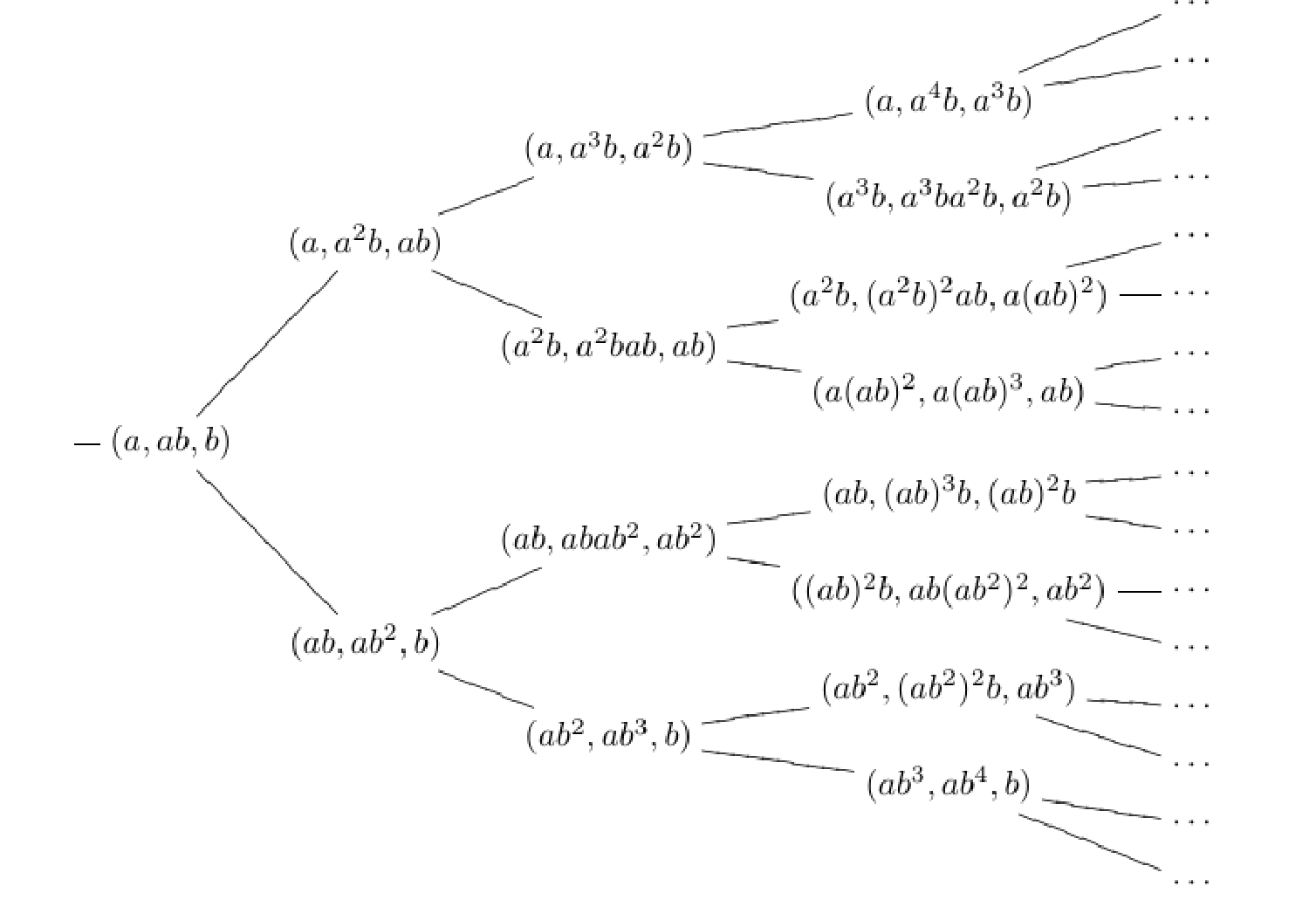} 
\caption{}\label{fig2}
\end{figure}

\begin{thm}(\cite{Bom}, \cite{Cohn1}, \cite{Cohn2})
For $A:=\begin{pmatrix}
2 & 1 \\
1 & 1 
\end{pmatrix}, 
B:=\begin{pmatrix}
5 & 2 \\
2 & 1 
\end{pmatrix}$ and triple of  Christoffel $ab$-word)

\noindent 
$(w(a,b), w(a,b)w'(a,b), w'(a,b))$, a triple of matrices $(w(A,B), w(A,B)w'(A,B), w'(A,B))$ is called triple of Cohn matrices.

\medskip

\noindent 
Then

\noindent 
$\frac{1}{3}(\mathrm{Tr}(w(A,B)), \mathrm{Tr}(w(A,B)w'(A,B)), \mathrm{Tr}(w'(A,B)))
=(w(A, B)_{1,2}, \{ w(A,B)w'(A,B)\}_{1,2}, w'(A,B)_{1,2})$ is a Markov triple.
\end{thm}

\noindent 
We obtain the following $q$-deformation of triple of Cohn matrices and triple of Markov numbers.

\noindent 
Put 
$A_q:= R_q L_q =
\begin{pmatrix}
q+1&q^{-1} \\  
1 & q^{-1}
\end{pmatrix}, ~
B_q:= R_q^2 L_q^2 =
\begin{pmatrix}
\frac{q^3+q^2+2q+1}{q} & \frac{q+1}{ q^2 } \\ 
\frac{q+1}{q} &q^{-2}
\end{pmatrix}
 \in SL(2, {\Bbb Z}[q,q^{-1}])$.

\begin{lem}
Let $(M,MN, N)$ be a triple of $q$-deformation of Cohn matrices corresponding Christoffel $ab$-word triple $(w_1(a,b), w_1 (a,b) w_2(a,b), w_2 (a,b))$ such that
\begin{equation}\label{eqn-10}
 (M, MN, N)=(w_1(A_q, B_q), w_1(A_q, B_q) w_2 (A_q , B_q), w_2 (A_q, B_q)).
\end{equation}
Then the following relation holds.

\begin{equation}\label{eqn-11}
\mathrm{Tr}(MNM^{-1}N^{-1})+2 =- \frac{(q-1)^2}{q^3}[3]_q^2
\end{equation}

\end{lem}

\begin{proof}
\noindent 
In general, the following relation holds for a matrix relation $P=XQ$, 
\begin{equation}\label{eqn-12}
\mathrm{Tr}([P,Q])=\mathrm{Tr}([P^{-1},Q^{-1}])=\mathrm{Tr}(P^{-1} Q^{-1} PQ) =\mathrm{Tr}(Q^{-1} X^{-1} Q^{-1} X Q Q)= \mathrm{Tr}(X^{-1} Q^{-1} XQ)=\mathrm{Tr}([X,Q]).
\end{equation}

\noindent 
Similarly, for $Q=PX$, $\mathrm{Tr}([P,Q])=\mathrm{Tr}([P,X])$.

\medskip

\noindent
By applying this relation to $q$-Markov Cohn matrices triple $(w(A_q, B_q), w(A_q, B_q) w'(A_q, B_q), w'(A_q, B_q))$ repeatedly, it is possible to reduce the initial relation to the simplest relation as follows:
\begin{equation}\label{eqn-1}
\mathrm{Tr}([w(A_q, B_q), w'(A_q, B_q)])+2=\mathrm{Tr}([A_q,B_q])+2=-\frac{(q-1)^2[3]_q^2}{q^3}
\end{equation}
\end{proof}

\medskip 

\noindent 
Similarly we can obtain the following Lemma form the proof of Lemma 2.2.

\begin{lem}
\noindent 
Let $(M, MN, N)$ be $q$-Markov Cohn matrices triple and put 
$X:=\begin{pmatrix}
m_{1,1} & m_{1,2} \\
m_{2,1} & m_{2,2} 
\end{pmatrix} \in \{ M , MN, N \}$

\noindent 
then we obtain 

\noindent 
(1)
\begin{equation}\label{eqn-1.1}
 \mathrm{Tr}(X)/[3]_q =m_{1,2} -(q-1) m_{2,2}  \in {\Bbb Z}[q, q^{-1}]
\end{equation}

\noindent 
(2)
\begin{equation}\label{eqn-.21}
 \mathrm{Tr}(X) \in [3]_q {\Bbb Z}[q,q^{-1}]
\end{equation}

\noindent 
(3) 
\begin{equation}\label{eqn-1.3}
m_{1,2}=(q-1)m_{2,2} +\mathrm{Tr}(X)/[3]_q .
\end{equation}

\end{lem}

\begin{thm}
Put 
$A_q:= 
\begin{pmatrix}
q+1&q^{-1} \\  
1 & q^{-1}
\end{pmatrix}, ~
B_q:=
\begin{pmatrix}
\frac{q^3+q^2+2q+1}{q} & \frac{q+1}{ q^2 } \\ 
\frac{q+1}{q} &q^{-2}
\end{pmatrix}
 \in SL(2, {\Bbb Z}[q,q^{-1}])$.

\noindent 
Then we have the followings:

\noindent 
For triples Christoffel $ab$-words  $(w(a,b), w(a,b) w'(a,b) , w'(a,b)), $

\noindent 
$
(x,y,z)=(\mathrm{Tr} w(A_q, B_q)/[3]_q ,  \{ \mathrm{Tr} \{w(A_q, B_q)w'(A_q , B_q ) \} \}/[3]_q , \mathrm{Tr} w'(A_q, B_q)/[3]_q ) $

\noindent 
$=:(h_{w(a,b)}(q), h_{w(a,b)w'(a,b)}(q), h_{w'(a,b)}(q)) \in {\Bbb Z}[q, q^{-1}]^3
$

 is a solution of the following equation:

\begin{equation}\label{eqn-14}
 x^2 + y^2 + z^2 +\frac{(q-1)^2}{q^3}= [3]_q xyz 
\end{equation}

\end{thm}

\begin{proof}

\noindent 
Fricke identity holds on $SL(2,{\Bbb Z}[q, q^{-1}]$ as follows:

\begin{equation}\label{eqn-15}
\mathrm{Tr}(X)^2+\mathrm{Tr}(XY)^2 +\mathrm{Tr}(Y)^2=\mathrm{Tr}(X) \mathrm{Tr}(XY) \mathrm{Tr}(Y) +\mathrm{Tr}(XYX^{-1}Y^{-1})+2
\end{equation}

\noindent 
 By using Lemma 2.2 and Lemma 2.3, we obtain

\noindent
$(x,y,z)=(\mathrm{Tr}(w(A_q, B_q)/[3]_q, \mathrm{Tr}(ww'(A_q,B_q))/[3]_q , \mathrm{Tr}(w'(A_q, B_q))/[3]_q) \in {\Bbb Z}[q,q^{-1}]^3$ is a solution of the $q$-deformed Markov equation $x^2 +y^2 +z^2 +\frac{(q-1)^2}{q^3}=[3]_q xyz$
\end{proof}

\begin{rem}
$(x,y,z)=([3]_q h_w (q), [3]_q f_{ww'}(q) ,[3]_q h_{w'}(q))$

\noindent
$= (\mathrm{Tr} w(A_q, B_q) ,  \mathrm{Tr} w(A_q, B_q)w'(A_q , B_q ) , \mathrm{Tr}(w'(A_q, B_q) ) \in ( [3]_q{\Bbb Z}[q, q^{-1}])^3$ is a solution of the equation 

\begin{equation}\label{eqn-16}
 x^2 + y^2 + z^2 +\frac{(q-1)^2[3]_q^2}{q^3}= xyz 
\end{equation}
\end{rem}


\begin{thm}
$(x,y,z)=(a_q, b_q, c_q)$ is a solution of 

\begin{equation}\label{eqn-17}
x^2 + y^2 + z^2 +\frac{(q-1)^2}{q^3}= [3]_q xyz 
\end{equation}

\medskip 

\noindent 
Then

$(\tilde{x},y,z)=([3]_q b_q c_q -a_q, b_q, c_q), (x, \tilde{y},z)=(a_q, [3]_q a_q c_q -b_q , c_q), (x,y,\tilde{z})=(a_q, b_q, [3]_q a_q b_q -c_q)$ are also solutions of the equation $(2.12)$
\end{thm}

\begin{proof}

$\tilde{x}^2 + y^2 +z^2 + \frac{(q-1)^2}{q^3}=([3]_qb_q c_q -a_q)^2 +b_q^2 +c_q^2 + \frac{(q-1)^2}{q^3} $

\noindent 
$=[3]q^2 b_q^2 c_q^2 -2[3]_q a_q b_q c_q +\{ a_q^2 +b_q^2 +c_q^2+ \frac{(q-1)^2}{q^3} \}$

\noindent 
$=[3]_q^2 b_q c_q -[3]_q a_q b_q c_q =[3]_q ([3]_q b_q c_q -a_q)b_q c_q =[3]_q \tilde{x}yz$

\noindent 
Then $(\tilde{x},y,z)$ is also a solution of $ (2.12)$ . Similarly $(x, \tilde{y}, z), (x,y, \tilde{z})$  are also solutions of $(2.12)$.
\end{proof}

\begin{rem}
In classical situation, a Markov triple $(a,b,c)$ is orthogonal to $(a-3bc, b,c)$ as 3-dimensional vectors. However 
for a solution $(a_q, b_q, c_q)$ of  the equation $(2.12)$, we consider usual inner product $( \cdot , \cdot)$ on ${\Bbb R}^3$.  
 Then, 
since $((a_q, b_q, c_q), (a_q -[3]_q b_q c_q , b_q , c_q))=a_q^2 +b_q^2 +c_q^2 -[3]_q a_q b_q c_q=-\frac{(q-1)^2}{q^3} \neq 0,$  $(a_{q}, b_{q}, c_{q})$ is not orthogonal to $(a_{q}- [3]_{q} b_{q} c_{q} )$ for usual inner product.However, since  $~~(\rightarrow 0~(q \rightarrow 0))$, 
the inner product is near to orthogonal.
\end{rem}

\begin{exam}
$
h_{a^3b}(q)^2 + h_{a^3b a^2b}(q)^2 + h_{a^2b}(q)^2 +\frac{(q-1)^2}{q^3}$

\noindent
$=\frac{\mathrm{tr}(A_q^3B_q)}{[3]_q}^2+\frac{\mathrm{tr}(A_q^3B_q A_q^2B_q)}{[3]_q}^2+\frac{\mathrm{tr}(A_q^2B_q)}{[3]_q}^2+\frac{(q-1)^2}{q^3}$.

\noindent 
$=\{\frac { \left( {q}^{2}+1 \right)  \left( {q}^{6}+3\,{q}^{5}+3\,{q}^{4
}+3\,{q}^{3}+3\,{q}^{2}+3\,q+1 \right) }{{q}^{5}} \}^2 $

\noindent 
$+
\{\frac { \left( {q}^{4}+{q}^{3}+{q}^{2}+q+1 \right)  \left( {q}^{12}+5
\,{q}^{11}+12\,{q}^{10}+22\,{q}^{9}+32\,{q}^{8}+39\,{q}^{7}+43\,{q}^{6
}+39\,{q}^{5}+32\,{q}^{4}+22\,{q}^{3}+12\,{q}^{2}+5\,q+1 \right) }{{q}^{9}}\}^2$

\noindent 
$+\{\frac {{q}^{4}+{q}^{3}+{q}^{2}+q+1}{{q}^{3}}\}^2
+\frac{(q-1)^2}{q^3}
$

\noindent 
$=[3]_q 
\{\frac { \left( {q}^{2}+1 \right)  \left( {q}^{6}+3\,{q}^{5}+3\,{q}^{4
}+3\,{q}^{3}+3\,{q}^{2}+3\,q+1 \right) }{{q}^{5}} \}$

\noindent 
$\{\frac { \left( {q}^{4}+{q}^{3}+{q}^{2}+q+1 \right)  \left( {q}^{12}+5
\,{q}^{11}+12\,{q}^{10}+22\,{q}^{9}+32\,{q}^{8}+39\,{q}^{7}+43\,{q}^{6
}+39\,{q}^{5}+32\,{q}^{4}+22\,{q}^{3}+12\,{q}^{2}+5\,q+1 \right) }{{q}^{9}}\}$

\noindent 
$\{\frac {{q}^{4}+{q}^{3}+{q}^{2}+q+1}{{q}^{3}}\}$

\noindent 
$=[3]_q \frac{\mathrm{tr}(A_q^3B_q)}{[3]_q} 
\frac{\mathrm{tr}(A_q^3B_q A_q^2B_q)}{[3]_q}  
\frac{\mathrm{tr}(A_q^2B_q)}{[3]_q} 
 =[3]_q h_{a^3b}(q) h_{a^3 b a^2 b}(q) h_{a^2 b}(q)$

\end{exam}

\subsection{An application of $q$-Markov triples to fixed points of associated Cohn matrices}
Here we introduce an application of $q$-deformation of Markov triples $(h_{w(a,b)}(q), h_{w(a,b)w'(a,b)}(q), h_{w'(a.b)}(q))$ to  fixed point of $q$-Cohn matrix $w(A_q, B_q)w'(A_q, B_q)$.

\noindent 
In Theorem 16 of \cite{Bom}, Bombieri proved a nice properties of Cohn matrix coming from Christoffel $ab$-words and applied to quadratic irrational number. The following theorem is related to $q$-quadratic irrational number and $q$-Cohn matrix $w(A_{q}, B_{q} )$ for Christoffel $ab$-word $w(a,b)$.

\begin{thm} 
For fixed point $\theta_q(w(a,b))$ of  linear fractional transformation with respect to   $q$-Cohn matrix $w(A_q  , B_q )$ corresponding to Christoffel $ab$-word $ w(a,b)$, we obtain

\begin{equation}\label{eqn-17}
\theta_q (w(a,b))=[ \overline{\Pi}]_q. 
\end{equation}

\noindent 
by using a finite sequence $\Pi$ that is obtained by substituting

\begin{equation}\label{eqn-18}
a \mapsto 1,1, ~~b \mapsto 2,2 .
\end{equation}

\end{thm}

\begin{proof}
From the definition of $q$-rational number $[\frac{r}{s}]_q =[a_1, \dots , a_{2m}]_q =[[c_1, \dots , c_k ]]_q $ and $\frac{w(A_q , B_q)_{1,1}}{w(A_q, B_q)_{2,1}}=[a_1, \dots , a_{2m}]_q $ due to  Theorem1.3 and Theorem1.4 , we can transform the relation $w(A_q, B_q) \cdot \theta_{q} (w(a,b))=\theta_{q} (w(a,b))$ to

\begin{equation}\label{eqn-fixedpoint}
\theta_{q} (w(a,b)):=
[a_1]_q  
+\cfrac{ q^{a_1}   }{ [a_2]_{ q^{-1} }  + \cfrac{ q^{-a_2}  }
{  [a_3]_q  + \cfrac{q^{a_3} }
{ [a_4]_{q^{-1} } + \cfrac{ q^{-a_4} }{  \cfrac{\ddots }
{ [a_{2m-1}]_q+ \cfrac{ q^{a_{2m-1}} }{ [a_{2m}]_{q^{-1}}   +\cfrac{q^{-2m}  }{\theta_{q} (w(a,b))}  }}}}}} 
\end{equation}

\noindent 
Then we obtain that $\theta (w(a,b))=[\overline{a_1, \dots , a_{2m}}]_q$

\noindent 
On the other hand, the definition of $A_q$ and $ B_q , $ $A_q =R_q L_q =
\begin{pmatrix}
[1,1]_q & q^{-1} \\ 
1 & q^{-1} 
\end{pmatrix} $ and $B_q =R_q^2 L_q^2 =
\begin{pmatrix} 
q^{-1}[2,2]_q & q^{-2} [2]_q \\
q^{-1}[2]_q & q^{-2} 
\end{pmatrix}$.
 Thus we have $\theta_q (w(a,b))=[\overline{\Pi} ]_q$, where $\theta_q (w(a,b))=[\overline{\Pi} ]_q, $ we here $\Pi$ is a finite sequence that obtained 
 by substituting $a \mapsto 1,1, ~~b \mapsto 2,2.$ 
\end{proof}

\begin{exam}
$\theta_q (a^2b)=[\overline{ 1,1,1,1,2,2}]_q$

$=\frac {{q}^{8}+3\,{q}^{7}+5\,{q}^{6}+7\,{q}^{5}+5\,{q}^{4}+3\,{q
}^{3}+{q}^{2}-q-1}{2q \left( {q}^{6}+2\,{q}^{5}+4\,{q}^{4}+4\,{q}^{3}+4\,{q}
^{2}+3\,q+1 \right) }$

\footnotesize{
$+\frac{\sqrt {{q}^{16}+6\,{q}^{15}+19\,{q}^{14}+44\,{q}^{13}+81\,{q}^{12}+126\,{q}^{11}+171\,{q}^{10}+204\,{q}^{9}+213\,{q}^{8}+
204\,{q}^{7}+171\,{q}^{6}+126\,{q}^{5}+81\,{q}^{4}+44\,{q}^{3}+19\,{q}^{2}+6\,q+1}}{2q \left( {q}^{6}+2\,{q}^{5}+4\,{q}^{4}+4\,{q}^{3}+4\,{q}^{2}+3\,q+1 \right)} $ }

\end{exam}

\begin{rem}
The $q$-polynomial 
${q}^{16}+6\,{q}^{15}+19\,{q}^{14}+44\,{q}^{13
}+81\,{q}^{12}+126\,{q}^{11}+171\,{q}^{10}+204\,{q}^{9}+213\,{q}^{8}+
204\,{q}^{7}+171\,{q}^{6}+126\,{q}^{5}+81\,{q}^{4}+44\,{q}^{3}+19\,{q}
^{2}+6\,q+1$in the square root is vertical symmetry at middle term $213q^8$.
\end{rem}

\begin{rem}
For $q$-Cohn matrix $C_q(w(a,b))=\begin{pmatrix}
 r_q & t_q \\
s_q & u_q 
\end{pmatrix}$ for Christoffel $ab$-word $w(a,b)$.

\noindent 
If $\frac{r_q}{s_q} \rightarrow \frac{r}{s}$ $(q \rightarrow 1)$, 
then we have 

$\frac{r}{s}=[ w(a,b)|_{a\mapsto 1,1,~ b \mapsto 2,2}]=[\Pi]$ and 
$\frac{r_q}{s_q}=[\frac{r}{s}]_q$.
\end{rem}

\subsection{Other $q$-deformations of Markov triples}

\begin{prop}
For $A_q:= 
\begin{pmatrix}
q+1&q^{-1} \\  
1 & q^{-1}
\end{pmatrix}, ~
B_q:=
\begin{pmatrix}
\frac{q^3+q^2+2q+1}{q} & \frac{q+1}{ q^2 } \\ 
\frac{q+1}{q} &q^{-2}
\end{pmatrix}
$
 
\noindent 
and Christoffel $ab$-word $w(a,b),w(a,b)w'(a,b), w(a,b))$,

\noindent 
$(x,y,z,x')=w(A_q, B_q)_{1,2}, \{ w(A_q, B_q)w'(A_q , B_q )\}_{1,2}, w'(A_q, B_q)_{1,2}, \mathrm{Tr}(w(A_q, B_q)/[3]_q) \in {\Bbb Z}[q, q^{-1}]^3$ is a solution of the equation:

\begin{equation}\label{eqn-20}
\frac{1}{q^3} x^2 + y^2 + z^2 = [3]_q x'yz 
\end{equation}

\noindent 
on $ {\Bbb Z}[q,q^{-1}]^3 . $
\end{prop}

\begin{prop}
For $A_q:= 
\begin{pmatrix}
q+1&q^{-1} \\  
1 & q^{-1}
\end{pmatrix}, ~
B_q:=
\begin{pmatrix}
\frac{q^3+q^2+2q+1}{q} & \frac{q+1}{ q^2 } \\ 
\frac{q+1}{q} &q^{-2}
\end{pmatrix}, 
$

\noindent 
$(x,y,z, y',z')=$

\noindent 
$
(h_{w_1(a,b)}(q), h_{w_1(a,b) w_2(a,b)}(q), h_{w_2(a,b)}(q), \left\{  (w_1(A_q, B_q)w_2(A_q, B_q)) \right\}_{1,2} |_{q \mapsto q^{-1} } \cdot q^{\textrm{deg.} h_{w_1(a,b)w_2(a,b) }(q)}, $ $w_2(A_q, B_q)_{1,2})$ is a solution of

\begin{equation}\label{eqn-21}
\frac{1}{q^{3}} x^2 +y^2 +z^2 + \frac{(q-1)^2}{q^3} =[3]_q x y' z' .
\end{equation}

\noindent 
where $h_{w(a,b)}:=\mathrm{Tr}w(A_q, B_q)/[3]_q$.
\end{prop}


\begin{exam}
For $(w_1(a,b), w_1(a,b)w_2(a,b), w_2 (a,b))=(a^3b , a^3b a^2b, a^2b)$  corresponding to a Markov triple $(34, 1325, 13)$,

\noindent 
$\begin{array}{l}
w_1 (A_q,B_q)_{1,2}={\frac {{q}^{7}+4\,{q}^{6}+6\,{q}^{5}+7\,{q}^{4}+7\,{q}^{3}+5\,{q}^{2}
+3\,q+1}{{q}^{5}}} \\
w_2(A_q,B_q)_{1,2}={\frac {{q}^{5}+3\,{q}^{4}+3\,{q}^{3}+3\,{q}^{2}+2\,q+1}{{q}^{4}}} \\
\{ w_1(A_q,B_q) w_2(A_q,B_q)\}_{1,2}:=q^{-9} ({q}^{15}+7\,{q}^{14}+23\,{q}^{13}+52\,{q}^{12}+93\,{q}^{11}+
138\,{q}^{10} \\
+177\,{q}^{9}+197\,{q}^{8}+194\,{q}^{7}+167\,{q}^{6}+125
\,{q}^{5}+81\,{q}^{4}+44\,{q}^{3}+19\,{q}^{2}+6\,q+1).
\end{array}$

\medskip

\noindent 
Then 

\medskip
 
\noindent 
$
\frac{1}{q^3} w_1 (A_q,B_q)_{1,2}^2 +\{w_1(A_q,B_q) w_2(A_q,B_q) \}_{1,2}^2+w_2(A_q,B_q)_{1,2}^2 $

\noindent 
$=q^{-18}  ( {q}^{2}+1 )  \left( {q}^{6}+3\,{q}^{5}+3\,{q}^{4}+3\,{q}
^{3}+3\,{q}^{2}+3\,q+1 \right) $

\noindent 
$ ( {q}^{15}+7\,{q}^{14}+23\,{q}^{13}+52\,{q}^{12}+93\,{q}^{11}+138\,{q}^{10}+177\,{q}^{9}+197\,{q}^{8}+194\,{q}^{7}$

\noindent 
$+167\,{q}^{6}+125\,{q}^{5}+81\,{q}^{4}+44\,{q}^{3}+19\,{q}
^{2}+6\,q+1 )  ( {q}^{2}+q+1 )  ( {q}^{5}+3\,{q}
^{4}+3\,{q}^{3}+3\,{q}^{2}+2\,q+1 ) $

\noindent 
$=[3]_q 
h_{a^3b}(q) \{w_1(A_q,B_q) w_2(A_q,B_q) \}_{1,2} w_2(A_q,B_q)_{1,2} .
$


\end{exam}

\begin{exam}
For Christoffel $ab$-word triple $(w_1 (a,b), w_1(a,b) w_2 (a,b), w_2 (a,b))$ $=(abab^2, $ $ abab^2ab^2, ab^2)$, let corresponding $q$-Cohn matrices triple be 
$(w_1(A_q , B_q), w_1(A_q, B_q) w_2(A_q, B_q),$ $ w_2 (A_q , B_q))=(A_q B_q A_q B_q^2 , A_q B_q A_q B_q^2 A_q B_q^2, A_q B_q^2)$, 

\noindent 
then 

\noindent 
$h_{abab^2}(q):=\mathrm{Tr}(A_q B_q A_q B_q ^2)/[3]_q
= q^{-8} ({q}^{14}+4\,{q}^{13}+11\,{q}^{12}+22\,{q}^{11}+36\,{q}^{10}+50
\,{q}^{9}+60\,{q}^{8}+65\,{q}^{7}+60\,{q}^{6}+50\,{q}^{5}+36\,{q}^{4}+
22\,{q}^{3}+11\,{q}^{2}+4\,q+1)
$

\noindent 
$h_{abab^2ab^2}(q):=\mathrm{Tr}(A_q B_q A_q B_q^2A_q B_q^2)/[3]_q =q^{-13}( {q}^{2}+1 )$

\medskip

\noindent 
$\cdot ( {q}^{22}+7\,{q}^{21}+29\,{q}
^{20}+87\,{q}^{19}+208\,{q}^{18}+417\,{q}^{17}+724\,{q}^{16}+1114\,{q}
^{15}+1540\,{q}^{14}+1931\,{q}^{13}+2206\,{q}^{12}+2305\,{q}^{11}+2206
\,{q}^{10}+1931\,{q}^{9}+1540\,{q}^{8}+1114\,{q}^{7}+724\,{q}^{6}+417
\,{q}^{5}+208\,{q}^{4}+87\,{q}^{3}+29\,{q}^{2}+7\,q+1 )$

\medskip

\noindent 
$h_{ab^2}(q):={\frac {{q}^{8}+2\,{q}^{7}+4\,{q}^{6}+5\,{q}^{5}+5\,{q}^{4}+5\,{q}^{3}
+4\,{q}^{2}+2\,q+1}{{q}^{5}}}
$

\noindent 
$\frac{1}{q^3} h_{abab^2}(q)^2 +h_{abab}(q)^2ab^2)^2 +h_{ab}(q)^2)^2+\frac{(q-1)^2}{q^3} $

\noindent 
$=q^{-26}  \left( {q}^{2}+q+1 \right)  \left( {q}^{2
}-q+1 \right)$

\noindent 
$ \cdot ( q+1) ( {q}^{7}+3\,{q}^{6}+5\,{q}^{5}+6\,{q}^{4}+6\,{q}^
{3}+5\,{q}^{2}+2\,q+1 )$

\noindent 
$\cdot ( {q}^{14}+4\,{q}^{13}+11\,{q}^{12
}+22\,{q}^{11}+36\,{q}^{10}+50\,{q}^{9}+60\,{q}^{8}+65\,{q}^{7}+60\,{q
}^{6}+50\,{q}^{5}+36\,{q}^{4}+22\,{q}^{3}+11\,{q}^{2}+4\,q+1 )$

\noindent 
$\cdot ( {q}^{22}+6\,{q}^{21}+25\,{q}^{20}+75\,{q}^{19}+184\,{q}^{18}+
378\,{q}^{17}+675\,{q}^{16}+1063\,{q}^{15}+1501\,{q}^{14}+1913\,{q}^{
13}+2217\,{q}^{12}+2341\,{q}^{11}+2255\,{q}^{10}+1982\,{q}^{9}+1582\,{
q}^{8}+1142\,{q}^{7}+738\,{q}^{6}+422\,{q}^{5}+209\,{q}^{4}+87\,{q}^{3
}+29\,{q}^{2}+7\,q+1 )
$

\noindent 
$=
[3]_q  h_{abab^2}(q) (q^{23} \cdot (A_q B_q A_q B_q^2 A_q  B_q^2)_{1,2})_{q \mapsto q^{-1}}  ) \cdot (A_q B_q^2)_{1,2}.$

\end{exam}



\medskip


\section{A relation of castling transformations of 3-dimensional prehomogeneous vector spaces  and Markov triplets and the generalization}

\subsection{Prehomogeneous vector spaces and castling transformations}

\noindent 
Let $G$ be a linear algebraic group and $\rho$ its rational representation on a finite dimensional vector space $V$, all defined over the complex number field ${\Bbb C}$.
If $V$ has a Zariski-dense $G$-orbit ${\Bbb O}$, we call the triplet $(G,\rho,V)$ a {\it prehomogeneous vector space} (abbreviated by {\it PV}).
In this case, we call $v \in {\Bbb O}$ a {\it generic point}, and the isotropy subgroup $G_{v}=\{ g \in G \mid \rho(g)v=v \}$ at $v$ is called a {\it generic isotropy subgroup}. We call a prehomogeneous vector space $(G, \rho, V)$ a {\it reductive} prehomogeneous vector space if $G$ is reductive.
Let $\rho : G \to GL(V)$ be a rational representation of a linear algebraic group $G$ on an $m$-dimensional vector space $V$ 
and let $n$ be a positive integer with $m>n$. 
A triplet $\mathcal{C}_{1}:=(G \times GL(n),  \rho \otimes \Lambda_{1}, V \otimes V(n))$ is a prehomogeneous vector space if and only if a triplet $\mathcal{C}_{2}:=(G \times GL(m-n),  \rho^{\ast} \otimes \Lambda_{1}, V^{\ast} \otimes V(m-n))$ is a prehomogeneous vector space. 
We say that $\mathcal{C}_{1}$ and $\mathcal{C}_{2}$ are the {\it castling transforms} of each other.
Two triplets are said to be {\it castling equivalent} if one is obtained from the other by a finite number of successive castling transformations.
(cf. \cite{SaKi}, \cite{KiB}). 
Assume that $(G,\rho,V)$ is a prehomogeneous vector space with a Zariski-dense $G$-orbit ${\Bbb O}$.
A non-zero rational function $f(v)$ on $V$ is called a {\it relative invariant} if there exists a rational character $\chi:G \to GL(1)$ satisfying $f(\rho(g)v)=\chi (g) f(v)$ for $g \in G$.
In this case, we write $f \leftrightarrow \chi$.
Let $S_{i}=\{ v \in V \mid f_{i}(v)=0 \}$ $(i=1, \ldots, l)$ be irreducible components of $S:=V \setminus {\Bbb O}$ with codimension one.
When $G$ is connected, these irreducible polynomials $f_{i}(v)$ $(i=1, \ldots, l)$ are algebraically independent relative invariants and any relative invariant $f(v)$ can be expressed uniquely as $f(v)=cf_{1}(v)^{m_{1}} \cdots f_{l}(v)^{m_{l}}$ with $c \in  {\Bbb C}^{\times}$ and $m_{1}, \ldots, m_{l} \in {\Bbb Z}$.
These $f_{i}(v)$ $(i=1, \ldots, l)$ are called the {\it basic relative invariants} of $(G, \rho, V)$.
\begin{exam}
\label{Example1.1}
Let a  triplet $(SO(3) \times GL(1) , \Lambda_1 \otimes \Lambda_1, V(3) \otimes V(1))$ be a 3-dimensional prehomogeneous vector space, and hence its castling transform $(SO(3) \times GL(2), \Lambda_1 \otimes \Lambda_1 , V(3) \otimes V(2))$ is also a prehomogeneous vector space. The later can be regarded as $(SO(3) \times SL(2) \times GL(1) , \Lambda_1 \otimes \Lambda_1 \otimes \Lambda_1 , V(3) \otimes V(2) \otimes V(1))$, and its castling transform is given by $(SO(3) \times SL(2) \times GL(5), \Lambda_1 \otimes \Lambda_1 \otimes \Lambda_1 , V(3) \otimes V(2) \otimes V(5))$. Moreover, two new castling transforms are obtained from this space as follows. 
One is the castling transform $(SO(3) \times SL(5) \times GL(13) , \Lambda_1 \otimes \Lambda_1 \otimes \Lambda_1 , V(3) \otimes V(5) \otimes V(13))$ when the space above is regarded as $(SO(3) \times SL(5) \times GL(2) , \Lambda_1 \otimes \Lambda_1 \otimes \Lambda_1 , V(3) \otimes V(5) \otimes V(2))$, 
 and the other is the castling transform $(SO(3) \times SL(2) \times SL(5) \times GL(29), \Lambda_1 \otimes \Lambda_1 \otimes \Lambda_1 \otimes \Lambda_1, V(3) \otimes V(2) \otimes V(5) \otimes V(29))$
  obtained when the space is regarded as $ (SO(3) \times SL(2) \times SL(5) \times GL(1) , \Lambda_1 \otimes \Lambda_1 \otimes \Lambda_1 \otimes \Lambda_1, V(3) \otimes V(2) \otimes V(5) \otimes V(1))$ and so on. Thus we have a tree of prehomogeneous vector spaces from a seed prehomogeneous vector space $(SO(3) \times GL(1) , \Lambda_1 \otimes \Lambda_1 , V(3) \otimes V(1))$.
\end{exam}

\noindent 
In general, we abbreviate the triplet $(SO(3) \times SL(m_1) \times \cdots \times SL(m_{n-1}) \times GL(m_n), \Lambda_1 \otimes \Lambda_1 \otimes \cdots  \Lambda_1 \otimes \Lambda_1, V(3) \otimes V(m_1) \otimes \cdots \otimes V(m_{n-1}) \otimes V(m_n))$ as $(3, m_1, \dots , m_{n-1},m_n)$. 
By successive castling transformations we can draw the tree as Figure 3.
\begin{rem}
We can discuss the same argument for another 3-dimensional PV $(G, \rho , V(3))$ in stead of $(SO(3) \times GL(1), \Lambda_1 \otimes \Lambda_1, V(3) \otimes V(1))$. For example, we can choose $(G, \rho , V(3))=(GL(2) , 2 \Lambda_1 , V(3))$ in stead of $(SO(3) \times GL(1) , \Lambda_1 \otimes \Lambda_1, V(3) \otimes V(1))$.
\end{rem}

\subsection{A relation between castling transformations of 3-dimensional prehomogeneous vector spaces and Markov triples }


\begin{figure}[htbp]
\centering 
\includegraphics[width=18cm]{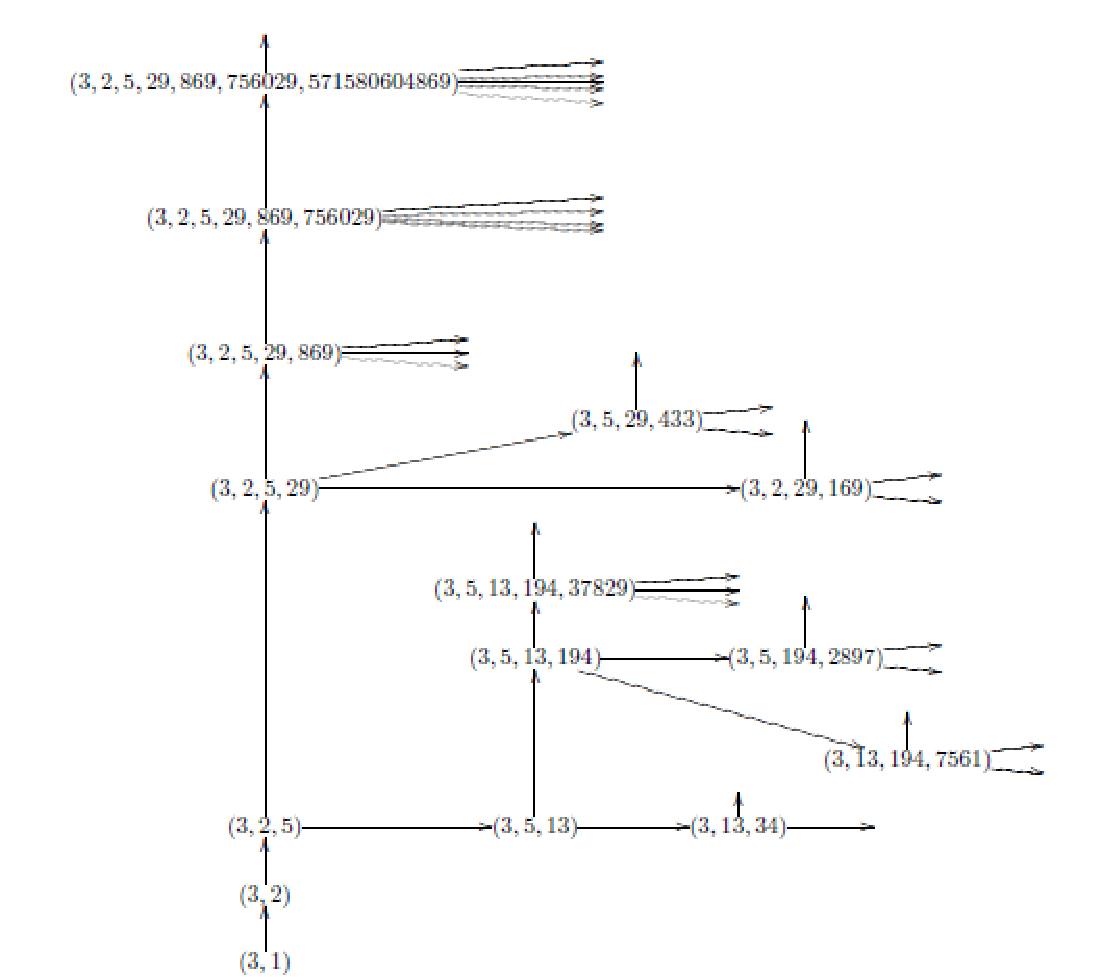} 
\caption{}\label{fig3}
\end{figure} 

\noindent 
Figure 3 is the tree of castling transforms from a seed PV $(G \times GL(1), \rho \otimes \Lambda_1 , V(3) \otimes V(1))$ for 3-dimensional PV $(G, \rho, V(3)).$ 
Let $(G, \rho , V(3))$ be a 3-dimensional PV. 
 In this section, we consider triples $(m_1, m_2, m_3)$ in PV $(3,m_1, m_2, m_3)$ which is castling transform of a $(3,1,1,2)=(G \times GL(1) \times GL(1) \times SL(2), \rho \otimes \Lambda_1 \otimes \Lambda_1 \otimes \Lambda_1, V(3) \otimes V(1) \otimes V(1) \otimes V(2))$. We explain that this triple $(m_1, m_2, m_3)$ of integers satisfies  a certain Diophantine equation.
Here we remark that subtree of 4-simple prehomogeneous vector spaces in the tree of castling transforms of the seed prehomogeneous vector space $(G \times GL(1), \rho \otimes \Lambda_1 , V(3) \otimes V(1))$. The branch  $(3,2,5,29)$, $(3,5,29,433)$, $(3,2,29,169)$, $(3,5,13,194)$, $(3,5,194,2897)$, $(3,13,194,7561)$, $\dots $ satisfy the following relations:

\noindent  
$\begin{array}{l} 
2^2 +5^2 +2 9^2 =870=3 \times 2 \times 5 \times 29, \\
5^2 +29^2 +433^2 =188355 =3 \times 5 \times 29 \times 433 , \\
2^2 +29^2 +169^2 =29406=3 \times 2 \times 29 \times 169 , \\
5^2 +13^2 +194^2 =37830 =3 \times 5 \times 13 \times 194, \\
13^2 +194^2 +7561^2 =57206526 =3 \times 13 \times 194 \times 7651 ,\\
\cdots
\end{array}$

\noindent  
In general, the following theorem holds.

\begin{thm}
\label{Theorem1}
Let $(G, \rho , V(3))$ be a 3-dimensional PV. For 4-simple prehomogeneous vector space $(G \times GL(m_1) \times SL(m_2) \times SL(m_3), \rho \otimes \Lambda_1 \otimes \Lambda_1 , V(3) \otimes V(m_1) \otimes V(m_2 ) \otimes V(m_3))$ (abbreviated by $(3, m_1, m_2, m_3) $) in the tree of castling transforms with respect to the seed prehomogeneous vector space $(G \times GL(1) , \rho \otimes \Lambda_1 , V(3) \otimes V(1))$, the triple of integers $(m_1, m_2 , m_3)$ satisfies the relation $m_1^2 + m_2^2 +m_3^2 =3 m_1 m_2 m_3$. Namely, $(m_1, m_2, m_3)$ is a Markov triple.
 
\end{thm}
\begin{proof}
 For 4-simple prehomogeneous vector space $(3,1,2,5)$, the triplet of integers $(1,2,5)$ satisfies the relation $1^2 +2^2 +5^2 =30 =3\times 1 \times 2 \times 5$.  
For 4-simple prehomogeneous vector space  $(3, m_1, m_2, m_3) $ ($ \mathrm{max}\{ m_1 , m_2 \} \leq m_3$) in the tree of castling transforms with respect to the seed prehomogeneous vector space $(G \times GL(1), \rho \otimes \Lambda_1  V(3) \otimes V(1))$,  by castling transform, we get two 4-simple prehomogeneous vector spaces : $(3, m_2, m_3, 3m_2 m_3 -m_1) , (3, m_1, m_3, 3m_1 m_3 -m_2)$ , here we put $m':=3 m_2 m_3 -m_1, m'':=3 m_1 m_3 -m_2$. 
Then , if $(m_1, m_2, m_3)$ satisfies a relation $m_1^2 +m_2^2 + m_3^2 =3 m_1 m_2 m_3$, then $(m_2, m_3, m')$ and $(m_1, m_3, m'')$ also satisfy the relation $m_2^2 +m_3^2 +{m'}^2 =3 m_2 m_3 m' $, $m_1^2 +m_3^2 +{m''}^2 =3 m_1 m_3 m'' $. 
In fact, $m_2^2 +m_3^2 +{m'}^2 $ $=m_2^2 m_3^2 +(3m_2 m_3 -m_1)^2$  $=m_2^2 m_3^2 +9m_2^2 m_3^2 +m_1^2 -6m_1 m_2 m_3 =$ 
$(m_1^2 m_2^2 +m_3^2) +9m_2^2 m_3^2 -6m_1 m_2 m_3$ 
$=3m_1 m_2 m_3 + 9 m_2^2 m_3^2 -6 m_1 m_2 m_3=3m_2 m_3 (3m_2 m_3 -m_1)$ 
$=3m_2 m_3 m'$. This $(m_1, m_2 , m_3)$ is a Markov triple.
\end{proof}

\subsection{ Prehomogeneous vector spaces parametrized by a positive fraction smaller than 1} 

From the view point of classification of prehomogeneous , the following application is interesting.

\begin{thm}
\label{Theorem2}
We can make a prehomogeneous vector space of Markov type from any positive reduced  fraction smaller than or equal to one. 
Conversely, any prehomogeneous vector space of Markov type comes from a reduced fraction smaller than or equal to 1.
\end{thm}
\begin{proof}
Form the one to one  correspondence between Farey triple $(\frac{r}{s}, \frac{r+r'}{s+s'}, \frac{r'}{s'})$ in Farey tree and Christoffel $ab$-word $(w(\frac{r}{s}), w(\frac{r}{s})w(\frac{r'}{s'}) , w(\frac{r'}{s'}))$ in tree of Christoffel $ab$-word in Figure 2 ( cf. \cite{Aig}). Moreover there exists one to one correspondence between Markov triples and Christoffel $ab$-words  by Theorem 2.1. Thus we obtain the correspondence between triples in Farey tree and Markov triples. Since any positive fraction which is  smaller than 1 or equal to 1 appear in second-entry in Farey triple in Farey tree and theorem 3.3, we have the correspondence between  the set of positive fraction smaller than or equal to one and the set of prehomogeneous vector spaces of the form  $(3, m_1, m_2, m_3)=(G \times  GL(m_1) \times GL( m_2 ) \times GL(m_3), \rho \otimes \Lambda_1 \otimes \Lambda_1 \otimes \Lambda_1, V(3) \otimes V(m_1) \otimes V(m_2) \otimes V(m_3))$.
\end{proof}


\medskip

\noindent 
Here we introduce an example.
A continued fraction is an expression obtained through an iterative process of representing a number as the sum of its integer part and the reciprocal of another number, then writing this other number as the sum of its integer part and another reciprocal, and so on. 

\begin{exam}
We take a reduced fraction ${8 \over {13}}$. 
Since ${8 \over {13}}=\displaystyle{ \frac{1}{1+\frac{1}{1+\frac{1}{1+\frac{1}{1+\frac{1}{1+1}}}}}  }=$ $[0;1,1,1,1,1+1]=[0;1,1,1,1,1,1]$, ${8 \over {13}}$ has neighbours $[0;1,1,1,1]={3 \over 5}$, $[0;1,1,1,1,1]={5 \over 8}$. Therefore ${ 8 \over {13}}={3 \over 5} \sharp {5 \over 8}$, where $\sharp$ means the Farey sum. Similarly, ${3 \over 5}={1 \over 2} \sharp {2 \over 3}$, ${5 \over 8}={3 \over 5} \sharp {2 \over 3}={1 \over 2} \sharp {2 \over 3} \sharp {2 \over 3}$. Here we know, that a reduced fraction ${1 \over p}$ (resp. ${{p-1} \over p}$) corresponds to Christoffel $ab$-word $a^{p-1}b$ (resp. $ab^{p-1}$). Therefore ${3 \over 5}$, ${5 \over 8}$ , ${8 \over {13}}$ corresponds to $ab$-word $abab^2$, $abab^2 ab^2$, $abab^2 abab^2ab^2$ respectively. Substituting matrices 
$A=\left( \begin{array}{cc}
2 &1 \\
1 & 1 
\end{array} \right) , B=
\left( \begin{array}{cc}
5 &2 \\
2 & 1 
\end{array} \right) $ to words $a,b$, then we have associated matrices 
$C[{3 \over 5}]:=ABAB^2 =\left( \begin{array}{cc}
1045 & 433 \\
613 & 254 
\end{array} \right),  
C[{5 \over 8}]:=ABAB^2 AB^2=\left( \begin{array}{cc}
90927 & 37666 \\
53266 & 22071 
\end{array} \right),
C[{8 \over {13}}]:=ABAB^2 ABAB^2 AB^2=
 \left( \begin {array}{cc} 118082927&48928105\\ \noalign{\medskip}
69267835&28701388\end {array} \right) 
$. We have a Markov triple $(433, 37666, 48928105)$ ( in fact $433^2 + 37666^2 +48928105^2=3 \times 433 \times 37666 \times 48928105$). 
Hence we have a prehomogeneous vector space 
$ (SO(3) \times GL(433) \times GL(37666) \times GL(48928105), V(3) \otimes V(433) \otimes V(37666) \otimes V(48928105))$ as a castling transform of a seed prehomogeneous vector space $(SO(3) \times GL(1), V(3) \otimes V(1))$. 
\end{exam}

\subsection{\bf Generalized cases}

In this subsection, we generalize the three-dimensional prehomogeneous vector spaces we have dealt with so far to the t-dimensional prehomogeneous vector spaces.
If we start from a prehomogeneous vector space with $t$-dimensional representation space $V(t)$ , namely, $(GL(1) \times G,  \Lambda_1 \otimes \rho, V(1)  \otimes V(t))$, we have the following tree in Figure 4. of sequences of polynomials.


\begin{figure}[htbp]
\centering 
\includegraphics[width=19cm]{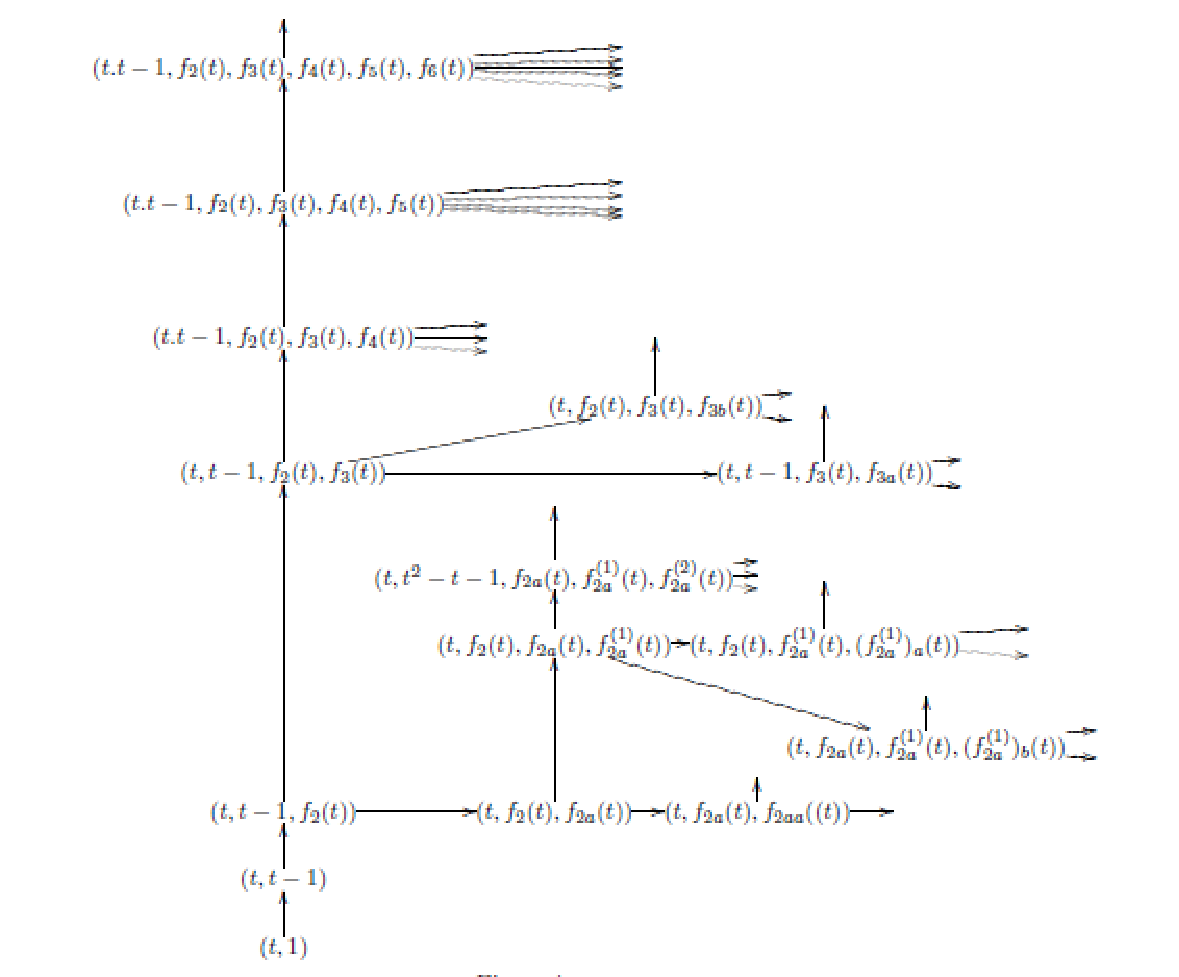} 
\caption{}\label{fig4}
\end{figure}

\noindent 
where 

\noindent 
$f_2 (t)=t^2-t-1$,

\noindent 
$f_3 (t)=t^4 -2t^3 +t-1$,

\noindent 
$f_4 (t)=t^7-3t^6+t^5+2t^4+t^3-t^2-t-1$,

\noindent 
$f_5 (t)={t}^{14}-6\,{t}^{13}+11\,{t}^{12}-2\,{t}^{11}-9\,{t}^{10}-4\,{t}^{9}+
10\,{t}^{8}+7\,{t}^{7}-2\,{t}^{6}-7\,{t}^{5}-3\,{t}^{4}+{t}^{3}+2\,{t}
^{2}+t-1$,

\noindent 
$f_6 (t) ={t}^{28}-12\,{t}^{27}+58\,{t}^{26}-136\,{t}^{25}+127\,{t}^{24}+56\,{t}
^{23}-126\,{t}^{22}-158\,{t}^{21}+229\,{t}^{20}+196\,{t}^{19}-158\,{t}
^{18}-314\,{t}^{17}+34\,{t}^{16}+294\,{t}^{15}+146\,{t}^{14}-142\,{t}^
{13}-213\,{t}^{12}-26\,{t}^{11}+116\,{t}^{10}+90\,{t}^{9}-9\,{t}^{8}-
45\,{t}^{7}-23\,{t}^{6}+5\,{t}^{5}+9\,{t}^{4}+3\,{t}^{3}-{t}^{2}-t-1$,

\noindent 
$f_{3a} (t)=t^5-2t^4+2t+1,t^5-2t^4-t^2+2t+1$,

\noindent 
$f_{2a}^{(1)} (t)=t^6-2t^5-2t^4+4t^3-t^2-t-1$, 

\noindent 
$(f_{2a}^{(1)})_a (t)= t^9-3t^8-t^7+8t^6-t^5-6t^4-2t^3+3t^2+3t-1$,

\noindent 
$(f_{2a}^{(1)})_b (t)= t^{10}-3t^9-2t^8+11t^7-t^6-12t^5+2t^4+3t^2+1$,

\noindent 
$f_{2a}^{(2)} (t)= t^{12}-4t^{11}+16t^9-10t^8-22t^7+15t^6+14t^5-5t^4-6t^3+t-1$,

\noindent 
$f_{3aa} (t) =t^4-t^3-3t^2+2t+1$,


\medskip

\noindent 
This tree is very complicated. We can recognize this tree of Castling transform from a seed prehomogeneous vector space $(G \times GL(1) , \rho \otimes \Lambda_1 , V(t) \otimes V(1))$ as following some ways.
On this tree, there exist the following two kinds of operations:

\noindent 
$CT_{up}:(t, f_1 (t), \dots , f_l (t)) \mapsto (t, f_1(t), \dots , f_l (t), t \prod_{i=1}^l f_i (t) -1)$ 

\noindent 
(This operation corresponds to $ \rotatebox[origin=c]{270}{$\leftarrow $} $ in Figure 4).

\noindent 
$CT_{flat_i}: (t,f_1 (t), \dots , f_{i-1}, f_i (t) , f_{i+1} (t) , \dots , f_r (t))$

\noindent 
$\mapsto (t,f_1 (t), \dots , f_{i-1}, t \prod_{j=1, j \neq i}^r f_j (t) -f_i (t) , f_{i+1} (t) , \dots , f_r (t))$

\noindent 
(This operation corresponds to $ \rightarrow $ in Figure 4).

\noindent 
We call the tree in Figure 4 {\it $t$-castling tree}.


\subsection{Castling Markov tree of $t$-castling tree}

\noindent 
In the tree in Figure 4, we choose "4-simple-subtree" as follows:

\noindent 
Starting from $$(t, f_{a} (t) , f_{ab}(t), f_{b}(t)):=(t,t^{2} -t-1, t-1),$$
 we consider subtree of the form 
$$(t, f_{w(a,b)}(t) , f_{w(a,b)w'(a,b)}(t), f_{w'(a,b)}(t)) $$
 that is parametrized by Christoffel $ab$-word $w(a,b), w(a,b)w' (a,b),$ $ w(a,b)$.

\noindent 
$CT_{flat_{b}} (t, f_{a} (t), f_{ab}(t) , f_{b} (t))=(t , f_{a} (t) , f_{ab}(t), t f_{a}(t) f_{ab}(t) -f_{b}(t) )$ 
\noindent
and we put 

\noindent 
$f_{a^{2}b}(t):=
t f_{a}(t) f_{ab}(t) -f_{b}(t) $ 

\noindent 
and we get a triplet $(f_{a}(t), f_{a^{2}b}(t) , f_{ab}(t))$ that is parametrized by triple of Christoffel $ab$-word $(
a, a^{2}b, ab)$.

\noindent 
$CT_{flat_{a}} (t, f_{a} (t), f_{ab}(t) , f_{b} (t))=(t , t f_{ab}(t) f_{b}(t) -f_{a}(t) , f_{ab}(t), f_{b}(t))$

\noindent 
and we put

\noindent 
$f_{ab^{2}}(t):=
t f_{a}(t) f_{ab}(t)- f_{b}(t) $

\noindent  
and we get a triplet $(f_{a}(t), f_{ab^{2}}(t) , f_{ab}(t))$ that is parametrized by triple of Christoffel $ab$-word $(
ab, ab^{2}, b)$. 
In general, we get 

$$(f_{w(a,b)}(t), f_{w(a,b)^{2}w'(a,b)}(t), f_{w(a,b)w'(a,b)}(t))$$
 and 
$$(f_{w(a,b)w'(a,b)}(t), f_{w(a,b) w'(a,b)^{2}}(t), f_{w'(a,b)}(t))$$
 as follows:

\noindent 
$CT_{flat_{w'(a,b)}}(t, f_{w(a,b)}(t), f_{w(a,b)w'(a,b)}(t) , f_{w' (a,b)}(t))$

\noindent
$=(t, f_{w(a,b)}(t), f_{w(a,b)w'(a,b)}(t) , t f_{w(a,b)} (t) f_{w(a,b)w'(a,b)}(t) -f_{w'(a,b)}(t))$ 

\noindent 
and we put $f_{w(a,b)^{2}w'(a,b)}(t) =tf_{w(a,b)}(t) f_{w(a,b)w'(a,b)}(t) -f_{w'(a,b)}(t)$ 
and get a triple 
$$(f_{w(a,b)}(t), f_{w(a,b)^{2}w'(a,b)}(t), f_{w(a,b)w'(a,b)}(t)).$$

\noindent 
$CT_{flat_{w(a,b)}}(t, f_{w(a,b)}(t), f_{w(a,b)w'(a,b)}(t) , f_{w' (a,b)}(t))
=(t, t f_{w(a,b)} (t) f_{w(a,b)w'(a,b)}(t) -f_{w'(a,b)}(t) ,f_{w(a,b)w'(a,b)}(t) ,$ $f_{w'(a,b)} (t)  )$ 

\noindent
and we put $f_{w(a,b)w'(a,b)^{2}}(t) =t f_{w(a,b)w'(a,b)} (t) f_{w'(a,b)}(t) -f_{w(a,b)}(t)$ 
and get a triple 
$$(f_{w(a,b)w'(a,b)}(t), f_{w(a,b)w'(a,b)^{2}}(t), f_{w'(a,b)}(t)) .$$

\noindent 
Thus we have a class $\{ (f_{w(a,b)}(t), f_{w(a,b)w'(a,b)}(t), f_{w'(a,b)}(t)) \}_{(w(a,b),w(a,b)w'(a,b), w'(a,b)) \textrm{is a triple of Christoffel words}}$ and tree Figure 5 of them. 

\noindent


\begin{figure}[htbp]
\centering 
\includegraphics[width=16cm]{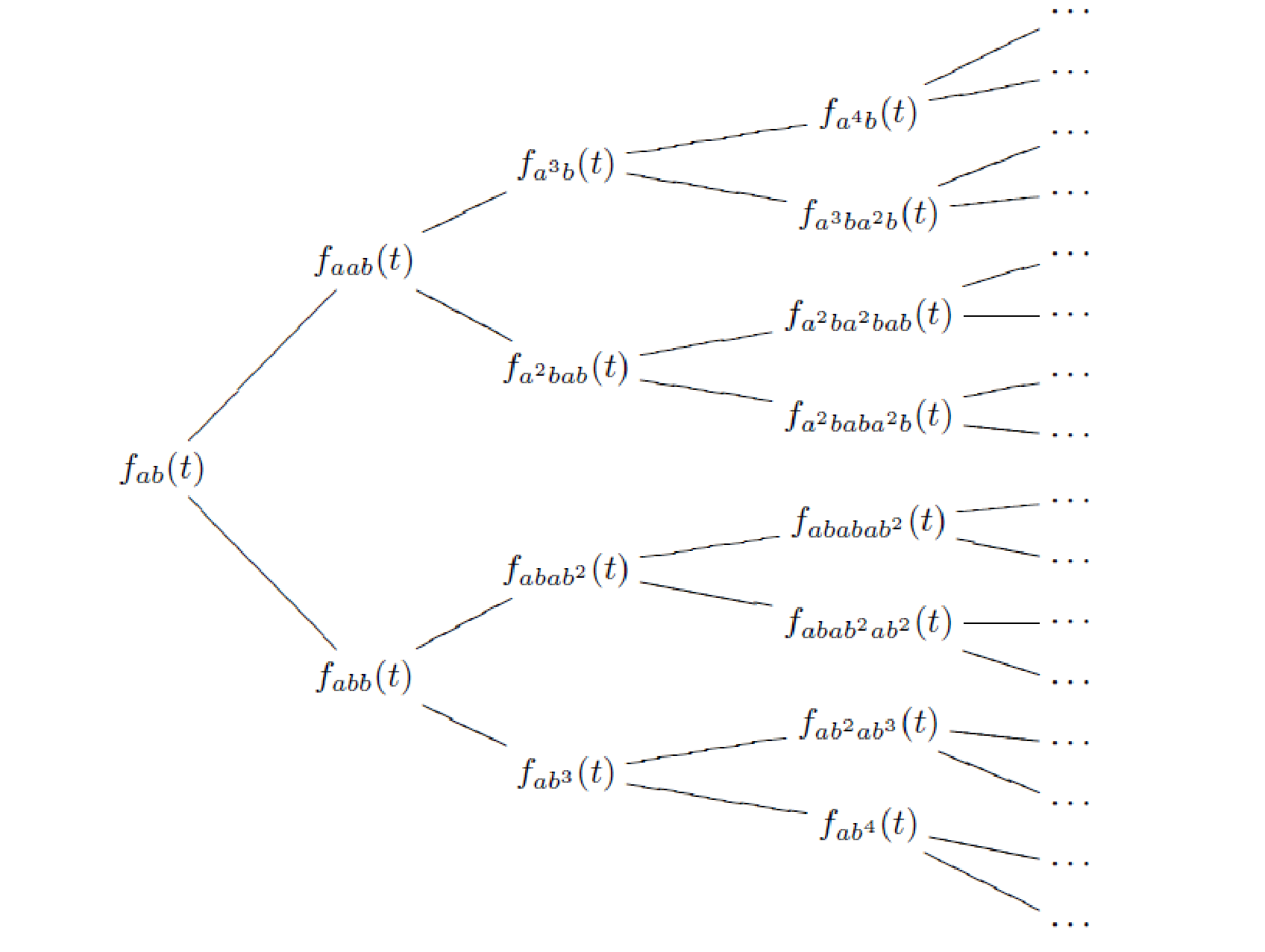} 
\caption{}\label{fig5}
\end{figure}

\medskip

\noindent
If we substitute $t=3$, we have the tree of Markov triples. Here we call these polynomials {\it castling-Markov polynomials} and call subtree 

\noindent 
 $\{ (f_{w(a,b)}(t), f_{w(a,b)w'(a,b)}(t), f_{w'(a,b)}(t)) \}_{(w(a,b),w(a,b)w'(a,b), w'(a,b)) \textrm{is a triple of Christoffel words}}$ {\it Castling Markov tree}. 
In specially, subtree of subsequence $\{ f_{a^{n} b} (t) \}$ satisfies the following recurrence relation:

\begin{equation}\label{eqn-22}
f_{a^{n+2} b} (t)=t f_{a^{n+1}b} (t)-f_{a^nb} (t). 
\end{equation}

\noindent 
Furthermore we have the following theorem.

\begin{thm} 
For a Christoffel $ab$-word triple $(w, ww' , w')$ where $ww'$ means a word that connects $w'$ after $w$,
let $(f_{w}(t), f_{ww'}(t) , f_{w'}(t))$ be triplet of castling Markov polynomials corresponding to $(w, w w' , w')$. 
Then $(f_{w} (t), f_{ww'}(t) , f_{w'} (t))$  is a solution of modified Markov equation 
\begin{equation}\label{eqn-23}
 x^{2} +y^{2} +z^{2} +(t-3)=txyz.
\end{equation}
\end{thm}

\begin{proof}
For a $t$-deformation $(a_{t}, b_{t}, c_{t}):=(f_{w(a,b)}(t), f_{w(a,b) w'(a,b)}(t), f_{w'(a,b)} )$ of a Markov triple $(a,b,c)=( \mathrm{Tr}w(A,B)/3, \mathrm{Tr} w(A,B) w'(A,B)/3, \mathrm{Tr} w'(A,B)/3)$, let two triples that come from $(a_{t}, b_{t}, c_{t})$ via castling transforms:

\begin{equation}\label{eqn-22.1}
\left\{
\begin{array}{l}
(a_{t}, \tilde{c}_{t}, b_{t} )=(f_{w(a,b)}(t) , f_{w(a,b)^{2} w' (a,b)} (t) , f_{w(a,b)w'(a,b)}(t)) =(a_{t}, ta_{t} b_{t} -c_{t} , b_{t} ) \\ 
(b_{t}, \tilde{a}_{t}, c_{t} )=(f_{w(a,b)w'(a,b)}(t) , f_{w(a,b) w' (a,b)^{2} } (t), f_{{w'(a,b)}(t)}) =(b_{t}, tb_{t} c_{t} -a_{t} , c_{t} ) 
\end{array} 
\right. 
\end{equation}

\noindent 
Then, since $a_{t}^{2}+ b_{t}^{2} +c_{t}^{2}+(t-3)$ $=ta_{t} b_{t} c_{t}$, 

\noindent 
$a_{t}^{2} +\tilde{c}_{t}^{2} +b^{2}_{t}+(t-3)$ $=
a_{t}^{2} +(ta_{t} b_{t} -c_{t} )^{2} +(t-3)=t^{2}a_{t}^{2} b_{t}^{2} -t a_{t} b_{t} c_{t} $ $=ta_{t}(ta_{t} b_{t} -c_{t}) b_{t}$ $=ta_{t} \tilde{c}_{t}b_{t}$, 

\noindent 
$b_{t }^{2} +\tilde{a}_{t}^{2} +c^{2}_{t}+(t-3)$ $=
b_{t}^{2} +(tb_{t} c_{t} -a_{t} )^{2} +(t-3)$ $=t^{2}b_{t}^{2}c_{t}^{2} -t a_{t} b_{t} c_{t}=b_{t} (tb_{t} c_{t} -a_{t})c_{t}$ $=tb_{t} \tilde{a}_{t}c_{t}$.
Therefore $(a_{t}, \tilde{c}_{t}, b_{t})$ and $(b_{t} , \tilde{a}_{t}, c_{t})$ are also solutions of $t$-Markov equation $x^{2} +y^{2} +z^{2}+(t-3)=txyz$.

\end{proof}

\subsection{ Chebyshev polynomials and castling Markov polynomials}

\noindent 
 The Chebyshev polynomials of the first kind are defined by the recurrence relation $T_0 (x)=1, T_1 (x)=x, T_{n+1} (x)=2 x T_n (x)-T_{n-1} (x)$. The Chebyshev polynomials of the second kind are defined by the recurrence relation $U_0 (x)=1, U_1 (x)=2x, U_{n+1} (x)=2 x U_n (x)-U_{n-1} (x)$. Modified Chebyshev polynomials of the first (resp. second) kind are defined by $t_n (x):=2T_n ({x \over 2})$ (resp. $u_n (x):=U_n ({x \over 2})).$ 
 From the property of Chebyshev polynomials, We obtain the following result.

\begin{prop}
Here if we put $S_n (t) :=f_{a^{n-1} b}(t)$ and $S_0(t):=1, S_1(t)=t-1$, 
we obtain the following:

\noindent 
(1) 
\begin{equation}\label{eqn-24}
\begin{array}{cc}
S_n (t)=\displaystyle{{1 \over {\sqrt{t^2 -4}} }}&
\biggl[\displaystyle{ (t-1) \left\{
\left( {{t+\sqrt{t^2-4}} \over 2} \right)^n 
- \left( {{t-\sqrt{t^2-4}} \over 2} \right)^n
\right\} } \\
 & \displaystyle{ -\left\{ 
\left( {{t+\sqrt{t^2-4}} \over 2} \right)^{n-1} 
- \left( {{t-\sqrt{t^2-4}} \over 2} \right)^{n-1}
\right\} } \biggr] 
\end{array}
\end{equation}

\noindent 
(2) 
\begin{equation}\label{eqn-25}
S_n (t)=t u_n (t) -u_{n-1} (t),
\end{equation}

\noindent 
where $u_n (t)$ is a Chebyshev polynomial of second type.

\noindent 
(3) 
\begin{equation}\label{eqn-26}
S_n (t):=\det 
\left( 
\begin{array}{ccccc}
t-1 & 1 & 0 & \cdots & 0 \\
1 & t & 1 & & \vdots  \\
0 & 1 & \ddots & \ddots & 0\\
 \vdots & \ddots  & \ddots & t &1  \\
 0 & \cdots   &   0  &  1 & t 
  \end{array} 
\right)~~(n \geq 1)
\end{equation}
\end{prop}

Furthermore we see the following property on the sub-sequences $\{ f_{ab^n} (t) \}$.

\begin{rem}
$f_{ab^n} (t)$ can be written by modified Chebyshev polynomials of first and second  kind as follows:

\begin{equation}\label{eqn-27}
f_{ab^n} (t)=({1 \over 2})t_n (t^2-t)+({1 \over 2})(t^2-t-2)u_{n-1}(t^2-t)=T_n ({1 \over 2}t(t-1))
+({1 \over 2})(t^2-t-2)U_{n-1}({1 \over 2}t(t-1)).
\end{equation}

If we put 

\begin{equation}\label{eqn-28}
\left\{ \begin{array}{c} 
p_n (t)={1 \over 2} t_n (t^2-t) +{1 \over 2}(t^2-t-2) u_{n-1} (t^2-t) \\
q_n (t) ={1 \over 2}t_n (t^2-t) +{1 \over 2}(t^2-t+2) u_{n-1} (t^2-t) 
\end{array} \right.
\end{equation}
 from properties of modified Chebyshev polynomials, we obtain 

\begin{equation}\label{eqn-28.5}
\left\{ \begin{array}{c} 
t_{n+2} (t)-t_n (t)=(t^4-4) u_n (t) \\
u_{n+2} (t) -u_n (t)=t_{n+2} (t) 
\end{array} \right. .
\end{equation}

\noindent 
Thus 
\begin{equation}\label{eqn-29}
\left\{ \begin{array}{c} 
p_{n+2} (t)-p_n (t) =(t^2-t-2) q_{n+1} (t) \\ 
q_{n+2} (t) -q_n (t) =(t^2-t+2) p_{n+1}(t) 
\end{array} \right. 
\end{equation}

\noindent 
hold.

\noindent 
For $f_{ a^{n-1} b } (t)=u_{n+1} (t)-u_n (t)$, $f_{a^n b}   (t)-f_{a^{ n-1 } b } (t)=t_{n+1} (t)$. If we put $g_n (t)=t_{ n+1} (t)-t_n (t)$, we have the following relation : 
$g_{ n+2 } (t)-g_n (t)=(t^4 -4)  (u_{ n-1} (t)- u_{n-2}  (t) )=(t^2-4) f_{ a^{n-3} b } (t) $.
\end{rem}

\medskip

\noindent 
{\bf 6.4. Continued fraction of polynomials and castling Markov polynomials}

We observe that castling-Markov polynomials are related to continued fraction of polynomials.
 For example,

\noindent 
$\displaystyle{\frac{f_{a^3ba^3ba^2b} (t)}{f_{a^3ba^3ba^3ba^2b} (t)}}$

\noindent 
$=\displaystyle{\frac{1}{{t}^{5}-{t}^{4}-3\,{t}^{3}+2\,{t}^{2}+t+\frac{1}{-({t}^{5}-{t}^{4}-3\,{t}^{3}+2\,{t}^{2}+t)+\frac{1}{-({t}^{3}-{t}^{2}-2\,t+1)}}}}$

\noindent 
\scalebox{0.9}[1.0]{
$= [0,{t}^{5}-{t}^{4}-3\,{t}^{3}+2\,{t}^{2}+t,-{t}^{5}
+{t}^{4}+3\,{t}^{3}-2\,{t}^{2}-t,{t}^{5}-{t}^{4}-3\,{t}^{3}+2\,{t}^{2}
+t,-{t}^{3}+{t}^{2}+2\,t-1] ,$}

\noindent 
$=[0;tf_{a^3b} (t),-tf_{a^3b}(t), tf_{a^3b} (t),-f_{a^2b}(t)]$

\noindent 
$\displaystyle{\frac{f_{a^2ba^2bab}(t)}{f_{a^2ba^2baba^2bab}(t)}=[0;tf_{a^2b} (t),-tf_{a^2bab}(t),t,-t,f_b (t)]} , $

\noindent
$\displaystyle{\frac{f_{a^2ba^2baba^2bab} (t)}{f_{a^2baba^2ba^2baba^2bab}(t)}=[0;tf_{a^2bab} (t),-tf_{a^2b}(t),tf_{a^2bab} (t),-t,t,-f_b (t)]}, $

\noindent 
$\displaystyle{\frac{f_{a^2babab}(t)}{f_{a2baba2babab}(t)}=[0;tf_{a^2bab} (t),-tf_{a^2bab}(t),-t,f_b (t)]},$ 

\noindent 
$\displaystyle{\frac{f_{a^2baba^2babab} (t)}{f_{a2baba2baba2babab}(t)}=[0;tf_{a^2bab} (t),-tf_{a^2bab} (t),tf_{a^2bab} (t),t,-f_b (t)]} , $

\noindent 
$ \displaystyle{\frac{f_{ababab^2} (t)}{f_{ababab^2abab^2} (t)}=[0;tf_{abab^2} (t),-tf_{abab^2} (t),f_{ab} (t),-f_{ab}(t)]}, $

\noindent 
$\displaystyle{\frac{f_{ababab^2abab^2} (t)}{f_{ababab^2abab^2abab^2}(t)}=[0;tf_{abab^2}(t),-tf_{abab^2}(t),tf_{abab^2} (t), -f_{ab} (t),f_{ab} (t)]}, $

$\begin{array}{l}
\displaystyle{\frac{f_{abab^2abab^2ab^2}(t)}{f_{abab^2abab^2abab^2ab^2}(t)}=[0;tf_{abab^2}(t),-tf_{abab^2}(t), tf_b (t),-f_{ab} (t)]}, \\
\displaystyle{\frac{f_{ab^2ab^3}(t)}{f_{ab^2ab^3ab^3}(t)}=[0;tf_{ab^3} (t),-tf_{ab^3} (t),f_{a^2b} (t),-f_b (t)]}, \\
\displaystyle{\frac{f_{ab^2ab^3ab^3} (t)}{f_{ab^2ab^3ab^3ab^3}(t)}
=[0;tf_{ab^2}(t),-tf_{ab^2} (t),tf_{ab^2}(t),-f_{ab^2}(t),-f_{a^2b}(t),f_b (t)]}, \\
\displaystyle{\frac{f_{ab^4}(t)}{f_{ab^3ab^4}(t)}=[0;tf_{ab^3} (t),-f_{ab^2} (t),-f_{ab}(t),f_{ab^2}(t)]}.
\end{array}$

\section{A relation between $q$-deformations and $t$-deformations}

 In this section, we introduce a relation between $q$-deformation of Markov triples and $t$-deformation of ones. $\mathrm{lim}_{q \mapsto 1} h_{w(a,b)} (q)=\frac{1}{3} \mathrm{Tr}(w(A,B)) =\mathrm{lim}_{t \mapsto 3} f_w (t)$ where $\frac{1}{3} \mathrm{Tr}(w(A,B))$ is a Markov number corresponding to Christoffel $ab$-word $w(a,b)$. Here we list up  $q$-deformations and $t$-deformations as follows:

\medskip

\noindent 
(i) List of $q$-deformations $\{ h_{w(a,b)}(q) \}_q$

\noindent 
$h_{a}(q)=\frac{1}{q}, $

\noindent 
$h_{b}(q)=q^{-2}(q^2 +1)$

\noindent 
$h_{ab}(q)=q^{-3} ({q}^{4}+{q}^{3}+{q}^{2}+q+1)  ,$

\noindent 
$h_{a^2b}(q)=q^{-4} ( {q}^{6}+2\,{q}^{5}+2\,{q}^{4}+3\,{q}^{3}+2\,{q}^{2}+2\,q+1)$

\noindent 
$h_{ab^2}(q)=q^{-5} ({q}^{8}+2\,{q}^{7}+4\,{q}^{6}+5\,{q}^{5}+5\,{q}^{4}+5\,{q}^{3}
+4\,{q}^{2}+2\,q+1)$

\noindent 
$h_{a^3b}(q)=q^{-5} ( {q}^{8}+3\,{q}^{7}+4\,{q}^{6}+6\,{q}^{5}+6\,{q}^{4}+6\,{q}^{3}
+4\,{q}^{2}+3\,q+1 ) , $

\noindent 
$h_{a^2bab}(q)=q^{-7}( {q}^{12}+4\,{q}^{11}+9\,{q}^{10}+16\,{q}^{9}+23\,{q}^{8}+29\,{
q}^{7}+30\,{q}^{6}+29\,{q}^{5}+23\,{q}^{4}+16\,{q}^{3}+9\,{q}^{2}+4\,q
+1) , $

\noindent 
$h_{abab^2}(q)=q^{-8}( {q}^{14}+4\,{q}^{13}+11\,{q}^{12}+22\,{q}^{11}+36\,{q}^{10}+50
\,{q}^{9}+60\,{q}^{8}+65\,{q}^{7}+60\,{q}^{6}+50\,{q}^{5}+36\,{q}^{4}+
22\,{q}^{3}+11\,{q}^{2}+4\,q+1), $

\noindent 
$h_{ab^3}(q)=q^{-7}( {q}^{12}+3\,{q}^{11}+8\,{q}^{10}+14\,{q}^{9}+20\,{q}^{8}+25\,{
q}^{7}+27\,{q}^{6}+25\,{q}^{5}+20\,{q}^{4}+14\,{q}^{3}+8\,{q}^{2}+3\,q
+1), $

\noindent 
$h_{a^4b}(q)=q^{-9} ({q}^{16}+4\,{q}^{15}+13\,{q}^{14}+29\,{q}^{13}+53\,{q}^{12}+82
\,{q}^{11}+110\,{q}^{10}+131\,{q}^{9}+139\,{q}^{8}+131\,{q}^{7}+110\,{
q}^{6}+82\,{q}^{5}+53\,{q}^{4}+29\,{q}^{3}+13\,{q}^{2}+4\,q+1) , $

\noindent 
$h_{a^3ba^2b}(q)=q^{-9}({q}^{16}+6\,{q}^{15}+18\,{q}^{14}+40\,{q}^{13}+72\,{q}^{12}+
110\,{q}^{11}+148\,{q}^{10}+175\,{q}^{9}+185\,{q}^{8}+175\,{q}^{7}+148
\,{q}^{6}+110\,{q}^{5}+72\,{q}^{4}+40\,{q}^{3}+18\,{q}^{2}+6\,q+1), $

\noindent 
$h_{a^2ba^2bab}(q)=q^{-11}( {q}^{20}+7\,{q}^{19}+26\,{q}^{18}+70\,{q}^{17}+151\,{q}^{16}+
276\,{q}^{15}+440\,{q}^{14}+623\,{q}^{13}+793\,{q}^{12}+914\,{q}^{11}+
959\,{q}^{10}+914\,{q}^{9}+793\,{q}^{8}+623\,{q}^{7}+440\,{q}^{6}+276
\,{q}^{5}+151\,{q}^{4}+70\,{q}^{3}+26\,{q}^{2}+7\,q+1), $

\noindent 
$h_{a^2babab}(q)=q^{-10}({q}^{18}+6\,{q}^{17}+20\,{q}^{16}+49\,{q}^{15}+97\,{q}^{14}+
164\,{q}^{13}+240\,{q}^{12}+313\,{q}^{11}+366\,{q}^{10}+385\,{q}^{9}+
366\,{q}^{8}+313\,{q}^{7}+240\,{q}^{6}+164\,{q}^{5}+97\,{q}^{4}+49\,{q
}^{3}+20\,{q}^{2}+6\,q+1), $

\noindent 
$h_{ababab^2}(q)=q^{-11}({q}^{20}+6\,{q}^{19}+22\,{q}^{18}+59\,{q}^{17}+128\,{q}^{16}+
235\,{q}^{15}+375\,{q}^{14}+533\,{q}^{13}+679\,{q}^{12}+784\,{q}^{11}+
822\,{q}^{10}+784\,{q}^{9}+679\,{q}^{8}+533\,{q}^{7}+375\,{q}^{6}+235
\,{q}^{5}+128\,{q}^{4}+59\,{q}^{3}+22\,{q}^{2}+6\,q+1), $

\noindent 
$h_{abab^2ab^2}(q)=q^{-13}( {q}^{24}+7\,{q}^{23}+30\,{q}^{22}+94\,{q}^{21}+237\,{q}^{20}+
504\,{q}^{19}+932\,{q}^{18}+1531\,{q}^{17}+2264\,{q}^{16}+3045\,{q}^{
15}+3746\,{q}^{14}+4236\,{q}^{13}+4412\,{q}^{12}+4236\,{q}^{11}+3746\,
{q}^{10}+3045\,{q}^{9}+2264\,{q}^{8}+1531\,{q}^{7}+932\,{q}^{6}+504\,{
q}^{5}+237\,{q}^{4}+94\,{q}^{3}+30\,{q}^{2}+7\,q+1), $

\noindent 
$h_{ab^2ab^3}(q)=q^{-12}( {q}^{22}+6\,{q}^{21}+24\,{q}^{20}+70\,{q}^{19}+165\,{q}^{18}+
328\,{q}^{17}+567\,{q}^{16}+870\,{q}^{15}+1201\,{q}^{14}+1504\,{q}^{13
}+1717\,{q}^{12}+1795\,{q}^{11}+1717\,{q}^{10}+1504\,{q}^{9}+1201\,{q}
^{8}+870\,{q}^{7}+567\,{q}^{6}+328\,{q}^{5}+165\,{q}^{4}+70\,{q}^{3}+
24\,{q}^{2}+6\,q+1) , $

\noindent 
$h_{ab^4}(q)=q^{-9}({q}^{16}+4\,{q}^{15}+13\,{q}^{14}+29\,{q}^{13}+53\,{q}^{12}+82
\,{q}^{11}+110\,{q}^{10}+131\,{q}^{9}+139\,{q}^{8}+131\,{q}^{7}+110\,{
q}^{6}+82\,{q}^{5}+53\,{q}^{4}+29\,{q}^{3}+13\,{q}^{2}+4\,q+1) , $


\medskip

\noindent 
(ii) List of $t$-deformations $\{ f_{w(a,b)}(t) \}_t$

$\begin{array}{l}
f_a (t)=1,  \\ 
f_b (t)=t-1 , \\ 
f_{ab}(t)=t^2-t-1,  \\
f_{a^2b}(t)=t^3-t^2-2t+1, \\
f_{ab^2}(t)=t^4-2t^3+t-1, \\
f_{a^3b}(t)=t^4-t^3-3t^2+2t+1, \\ 
f_{a^2bab}(t)=t^6-2t^5-2t^4+4t^3+t^2-t-1, \\ 
f_{abab^2}(t)=t^7-3t^6+t^5+3t^4-2t^3+1, \\  
f_{ab^3}(t)=t^6-3t^5+2t^4+t^3-3t^2+2t+1, \\ 
f_{a^4b}(t)=t^5-t^4-4t^3+3t^2+3t-1, \\
f_{a^3ba^2b}(t)=t^8-2t^7-4t^6+8t^5+4t^4-8t^3+t-1, \\
f_{a^2ba^2bab}(t)=t^{10}-3t^9-2t^8+11t^7-t^6-12t^5+2t^4+4t^3+1, \\
f_{a^2babab}(t)=t^9-3t^8-t^7+8t^6-t^5-6t^4-2t^3+3t^2+3t-1, \\ 
f_{ababab^2}(t)=t^{10}-4t^9+3t^8+5t^7-6t^6-t^5+t^4+3t^3-t^2-2t+1, \\ 
f_{abab^2ab^2}(t)=t^{12}-5t^{11}+7t^{10}+2t^9-12t^8+8t^7+2t^6-4t^5+1, \\
f_{ab^2ab^3}(t)=t^{11}-5t^{10}+8t^9-2t^8-9t^7+13t^6-4t^5-6t^4+5t^3-t^2-2t+1, \\
f_{ab^4}(t)=t^8-4t^7+5t^6-t^5-5t^4+7t^3-t^2-2t+1, 
\end{array}$

\medskip

\begin{thm}

\noindent 
(i)For Christoffel $ab$-word $w$, let $f_{w}(t)$ be $t$-deformation of Markov number corresponding to $w$ and  $h_{w}(q)$ be $q$-deformation of   Markov number corresponding to $w$.

\noindent 
Then the following relation holds:
\begin{equation}\label{eqn-40}
f_{w}([3]_q /q)=q h_{w}(q). 
\end{equation}

\noindent 
(ii) A set of $q$-deformations and  a set of $t$-deformations have a one to one correspondence.

\end{thm}

\begin{proof}
(i) From Theorem 2.5 and the construction of $(f_w (t) , f_{ww'} (t), f_{w'}(t))$ and $(h_w (q), h_{ww'} (q), h_{w'}(q))$, we obtain $f([q^{-1}[3]_q )=q h_w (q)$.

\noindent 
(ii)Let $(x,y,z)=(f_w (t), f_{ww'} (t) , f_{w'}(t) $ be a solution of $x^2 + y^2 +z^2 +(t-3) =txyz.$ 
Then put $t=q^{-1}[3]_q $ and deform the equation as follows:
$x^2 +y^2 +z^2 +\frac{[3]_q -3q}{q}=q^{-1} [3]_q xyz  \Leftrightarrow 
(x,y,z)=(q \tilde{x}, q \tilde{y}, q\tilde{z}), ~~(q\tilde{x})^2 +(q \tilde{y})^2 +( q \tilde{z})^2 =q^{-1}[3]_q (q \tilde{x})(q \tilde{y})(q \tilde{z})
 \Leftrightarrow 
q^2 \{ \tilde{x}^2 + \tilde{y}^2 +\tilde{z}^2 \}
= 
q^2 [3]_q  \tilde{x} \tilde{y} \tilde{z} \Leftrightarrow  
\tilde{x}^2 + \tilde{y}^2 +\tilde{z}^2
= 
 [3]_q    \tilde{x} \tilde{y} \tilde{z}.$
 Here from (i) $f_w (q^{-1} [3]_q)=q h_w (q)$, we have 

\noindent 
$(q h_w (q))^2 +(q h_{w w' }(q))^2 +(q h_{w'}( q))^2 +\frac{(1-q)^2}{q^3} =q^{-1} [3]_q 
(q h_w (q)) (q h_{ww'} (q)) (q h_{w'} (q)).$

\noindent 
This means that $(x,y,z)=(h_w (q), h_{ww'} (q) , h_{w'} (q)) $ is a solution of 
$x^2 +y^2 +z^2 +\frac{ (1-q)^2 }{ q^3 } =[3]_q xyz$.

\noindent 
On the other hand,
let $(x,y,z)=(h_w (q), h_{w w'} (q) , h_{w'} (q)) $ be a solution of 
 $x^2 + y^2 +z^2 +\frac{(1-q)^2}{q^3} =[3]_q xyz. $ 
Since $ \frac{(q-1)^2}{q^3}=\frac{q^{-1}[3]_q -3 }{ q^2}$ and (i) $f_w (q^{-1}[3]_q )=q h_w (q)$, 

\noindent 
$(q h_w (q))^2 +(q h_{w w'} (q))^2 +(q h_{w'} (q))^2 +(q^{-1}[3]_q -3) =
q^{-1} [3]_q  (q h_w  (q) )(q h_{w w'} (q)) (q h_{w'} (q))$

\noindent 
$\Leftrightarrow $

\noindent 
$f_w (t)^2 +f_{ww'} (t)^2 +f_{w'}(t)^2 +(t-3) =t f_w (t) f_{ww'} (t) f_{w'}(t).$

\noindent 
This means that $(x,y,z)=(f_w (t), f_{ww'}(t) , f_{w'} (t)) $ is a solution of 
$x^2 +y^2 +z^2 +(t-3) =t xyz$. 
\end{proof}

\medskip

\begin{exam}
(i)$f_{a^2b} (q^{-1} [3]_q)=q^{-3}[3]_q ^3 -q^{-2}[3]_q^2 -2 q^{-1}[3]_q +1$

\noindent 
$=q^{-3} \{ q^6 +2q^5 +2q^4 +3q^3 +2q^2 +2q +1 \} $

\noindent 
$=q h_{a^2b} (q)$

\medskip

\noindent 
(ii)$q h_{a^2bab} (q)=q^{-6} (q^{12} +4q^{11} +9q^{10} +16 q^9 +23 q^8  + 29 q^7 +30q^6 +29 q^5 +23q^4 +16q^3 +9q^2 +4q+1)$

\noindent 
$=(q^{-1}[3]_q )^6 -2 (q^{-1}[3]_q)^5 -2 (q^{-1}[3]_q)^4 +4(q^{-1}[3]_q)^3 +(q^{-1}[3]_q)^2 -(q^{-1}[3]_q )-1$

\noindent 
$=t^6 -2t^5 -2t^4 +4t^3 +t^2 -t-1$

\noindent 
$=f_{a^2bab}(t).$

\end{exam}

\medskip

\noindent 
{\bf Problem}

\noindent 
(1) In classical theory, a Markov triple gives the best approximation of a quadratic irrational to rational. Can $q$-deformations also say something about $q$-quadratic irrational number?

\noindent 
(2)
There is a relation  between $q$-deformations and $t$-deformations, then is there a deep relation between quantum topology , analytic number theory and prehomogeneous vector spaces?


\end{document}